\documentclass[final,english,leqno]{siamltex}
\usepackage[T1]{fontenc}
\usepackage[latin9]{inputenc}
\usepackage{verbatim}
\usepackage{bbm}
\usepackage{amsmath}

\usepackage{amssymb}
\usepackage{graphicx}
\usepackage{color}
\makeatletter

  \newtheorem{rem}{\protect\remarkname}
    
    \newtheorem{main}{Summary of our  Main Result}

\newcommand{\norm}[1]{\left\Vert#1\right\Vert}
\@ifundefined{showcaptionsetup}{}{%
 \PassOptionsToPackage{caption=false}{subfig}}
\usepackage{subfig}
\makeatother

\usepackage{babel}

\providecommand{\remarkname}{Remark}

\title{Dimension-Independent MCMC Sampling for Inverse Problems
with Non-Gaussian Priors}
\author{Sebastian J. Vollmer\thanks{Department of Statistics, 1 South Parks Road, Oxford OX1 3TG, sebastian.vollmer@stats.ox.ac.uk}}
\begin{document}
\maketitle
\begin{abstract}
The computational complexity of Markov chain Monte Carlo methods for the exploration of complex probability measures is a challenging and important problem both in statistics and the applied sciences. A challenge of particular importance arises in Bayesian inverse problems where the target distribution may be supported on an infinite dimensional state space. In practice this involves the approximation of the infinite dimensional target measure defined on sequences of spaces of increasing dimension bearing the risk of an increase of the computational error. Previous results have established dimension-independent bounds on the Monte-Carlo error of MCMC sampling for Gaussian prior measures. We extend these results by providing a simple recipe to obtain these bounds also in the case of non-Gaussian prior measures and by studying the design of proposal chains for the Metropolis-Hastings algorithm with dimension independent performance. This study is motivated by an elliptic inverse problem with non-Gaussian prior that arises in groundwater flow. We explicitly construct an efficient Metropolis-Hastings proposal based on local proposals in this case and we provide numerical evidence supporting the theory.


\end{abstract}



\pagestyle{myheadings}
\thispagestyle{plain}
\markboth{Sebastian J. Vollmer}{Dimension-Independent MCMC for Inverse Problems}

\global\long\def\data{y}
\global\long\def\noise{\eta}
\global\long\def\input{a}
\global\long\def\truth{a^{\dagger}}

\global\long\def\prior{{\mu_{0}}}

\global\long\def\posterior{\mu}

\global\long\def\priorCov{\mathcal{C}_{0}}

\global\long\def\nrOfObs{k}

\global\long\def\MHK{P}

\global\long\def\obsCov{\mathcal{C}_{1}}

\global\long\def\dcoef{a}
\global\long\def\disObs{d}

\global\long\def\spInput{X}

\global\long\def\spResponse{Y}

\global\long\def\opObs{\mathcal{G}}

\global\long\def\betterTrace{\sigma_{0}}

\global\long\def\noiseMeasure{\mu_{\noise}}

\global\long\def\nrOfObs{k}

\global\long\def\obsCov{\Gamma}

\global\long\def\pressure{p}

\global\long\def\pressureSpace{V}

\global\long\def\refMeasure{\mu_{0}}

\global\long\def\proposal{Q}

\global\long\def\target{\mu}

\global\long\def\MHK{P}

\global\long\def\densityTarRef{L}
\global\long\def\LOmu{{L_{0}^{2}(\mu)}}
\global\long\def\LOmuO{{L_{0}^{2}(\prior)}}

\newcommand{\lamMax}[2]{\lambda_{\text{max}}^{#1}\left({#2}\right)}
\newcommand{\lamMin}[2]{\lambda_{\text{min}}^{#1}\left({#2}\right)}

\section{Introduction}

The idea of the Bayesian approach to inverse problems is based on the assumption that not all parameter choices are a priori equally likely. Instead, the a priori knowledge about these parameters is modelled as a probability distribution - called the prior. By specifying the distribution of the noise, the parameters and the observed data can then be treated as jointly distributed random variables. Under certain conditions on the prior, model and noise, there exists a unique conditional distribution of the parameters given the data. This distribution is called the posterior and is an update of the prior using the data.  In this way uncertainty can be quantified using the posterior variance or the posterior probability of a set in the parameter space. Usually, the posterior is only expressed implicitly as an unnormalised density with respect to the prior. For this reason Monte-Carlo algorithms are used to approximate posterior expectations. In most cases it is not possible to generate i.i.d. samples from the posterior, instead correlated samples of a Markov chain are used. These algorithms are therefore called Markov-Chain-Monte-Carlo (MCMC) algorithms. For a recent review on Markov Chain Monte Carlo (MCMC) algorithms we refer the reader to \cite{brooks2011handbook}. It is well-known that under appropriate assumptions the error of MCMC algorithms is of order $\mathcal{O}(n^-{\frac{1}{2}})$ in the steps of the algorithms - the same as standard Monte-Carlo methods. However, we are interested in the Bayesian approach to the full infinite-dimensional inverse problem which is reviewed in Section \ref{sec:BIP}. In the infinite-dimensional approach to Bayesian Inverse Problems, the target measure $\posterior$ is a Borel measure on the Banach space $X$. It is given by 
\begin{equation}
\posterior \propto\densityTarRef\, \prior \label{eq:target}
\end{equation}
where $\refMeasure$ is a reference Borel probability measure on $X$ and $L$ is a density.We make no assumptions on the dimensionality of the space $X$ which is why our considerations also apply to infinite dimensional spaces.
 Usually, $\prior$ is the prior and $L$ is proportional to the likelihood.

Even under the restrictive assumption that 
\[
	0<L_\star < L < L^\star<\infty \label{ass:L},
\]
this has only been shown for the simple independence sampler (IS) \cite{2013HoangComplexityMCMC}. The IS algorithm is an Metropolis-Hastings algorithm generating independent proposals from a fixed distribution. This choice of proposals leads to a poor performance especially if the posterior is concentrated.  In \cite{2013HoangComplexityMCMC} results from \cite{Meyn:2009uqa} are used in order to show that the dimension independent convergence property assuming that the Markov chain associated with the Metropolis-Hastings algorithm is $\phi$-irreducible. On function spaces, this condition seems only to be verifiable in special cases such as the IS algorithm.  Nevertheless its performance does not depend on the dimension. We show that the same is true for a large class of  Metropolis-Hastings algorithms after a slight modification. The result is formulated in terms of $L^2$-spectral gaps. The concept of $L^2$-spectral gaps is introduced in Section \ref{sec:summaryMain} as well as a summary of our main results. In Section \ref{sec:overviewNonGaussian}, we explain how this result can be used for different non-Gaussian priors. We concentrate mainly on uniform series priors. We use these priors for our guiding  example - the inverse problem of reconstructing the diffusion coefficient from noisy measurements of the pressure in a Darcy model of groundwater flow. The underlying continuum model then corresponds to a linear elliptic PDE in divergence form, see Sections \ref{sec:Application} and  \ref{sec:num}.\\

Appropriate reviews on are contained in \cite{MR2102218}. The former is a key reference as the Bayesian approach is applied to real world applications using MCMC and optimisation techniques. This reference shows that the resulting methods can compete with state-of-the regularisation techniques in, for example, dental X-ray imaging. Whereas this reference applies the Bayesian method to a discretised version of the inverse problem, the survey article \cite{MR2652785} concerns the Bayesian approach to the full infinite dimensional problem which was originally developed in \cite{cotter-mcmc}. This approach was also taken independently in \cite{lasanen2002Disc,Lasanen2007}.
 
\subsection{Outline}
In Section \ref{sec:overviewNonGaussian}, we introduce the crucial concepts of $L^2$-spectral gaps and lazy chains and we give a brief summary of our main result (Theorem \ref{thm:main}). Subsequently, we discuss how proposals that satisfy the conditions of our main result can be chosen for uniform series and Besov priors in Section \ref{sec:overviewNonGaussian}. In Section \ref{sec:review},
we give a brief exposition to Bayesian inverse problems and Metropolis-Hastings algorithms on general state spaces before reviewing the implications of $L^2$-spectral gaps for the sample average and proving our main theorem using Dirichlet forms. In Section \ref{sec:Application}, we introduce an elliptic inverse problem that received much attention lately. Based on our programme presented in Section \ref{sec:review}, we construct a new class of sampling algorithms for this particular problem, the Reflection Random Walk Metropolis that we denote by RRWM and which satisfies the conditions of our main theorem. In Section \ref{sec:num}, we compare the RRWM, the standard Random Walk Metropolis (RWM) and the IS algorithms using numerical simulations for the posterior arising from this particular inverse problem.

\subsection{$L^2$-Spectral Gaps and a summary of our Main Result}\label{sec:summaryMain}

Let $P$  be a transition kernel of an MCMC algorithm with  invariant measure $\mu$.  The expectation of interest $\mathbb{E}_\mu f$ is approximated by the sample average
\begin{equation}
S_{n}(f)=\frac{1}{n}\sum_{i=1}^{n}f(X_{i}).
\end{equation}
If $S_{n}(f)$ satisfies a CLT, that is
\[
\sqrt{n}\left (S_{n}(f)-\mu(f)\right ) \rightharpoonup \mathcal{N}(0,\sigma_{f,P}^2),
\]
then the constant in $\mathcal{O}(n^{-{\frac{1}{2}}})$ is asymptotically normally distributed. The key quantity is then its asymptotic variance $\sigma_{f,P}^2$  because the smaller it is the better the MCMC kernel is for the function $f$ at hand. We consider Markov chains that are reversible as we are dealing with Metropolis-Hastings algorithms (see Section \ref{sec:MCMCGeneral}). In this case the  $L^2$-spectral gap $1-\beta$ is the right notion of convergence of the Markov chain because it guarantees a CLT as well as giving the bound 
 \[
 \sigma_{f,P}^{2} \leq \frac{2}{1-\beta} \sigma^2_{f,\mu}<\infty
 \]
where $\sigma^2_{f,\mu}$ is the standard deviation of $f$ with respect to $\mu$ (see Section \ref{sec:impl} for more details).  This means roughly that $\frac{2}{1-\beta}$ more steps of the Markov chain are needed than i.i.d. samples from $\mu$ to get the Monte-Carlo error to a similar or better level. 
For a sequence of transition kernels and target measures the dependence on, for example, the dimension $d$ can be quantified by the dependence of the constant in $\mathcal{O}(n^{-\frac{1}{2}})$ on $d$. For a sequence of transition kernels and target measures the dependence on, for example the dimension $d$, can be quantified by the dependence of the constant in $\mathcal{O}(n^{-\frac{1}{2}})$ on $d$. A dimension independent bound on the Monte Carlo error can therefore be obtained by proving a dimension independent lower bound on the spectral gap. 

In order to define {$L^{2}$-spectral gaps, recall that a Markov kernel $P$ with invariant measure $\mu$ acts on  $L^2_\mu(X)$ by $$Pf(x)=\int_X P(x,dy)f(y).$$
Jensen's inequality implies that the spectrum $\lambda(P)$ of $P$ is contained in the unit disk. If  $P$ is reversible, the corresponding Markov operator is self-adjoint. Thus,  the spectrum is real valued and $\lambda(P)\subseteq [-1,1]$.  Moreover, $P$ does always have  $1$ as  an eigenvalue since $P1=1$. 
For a self-adjoint operator on a Hilbert space $H$ we define the sharp lower $\lamMin{H}{A}$ and upper bound $\lamMax{H}{A}$ on the spectrum of the operator $A$ on the space $H$ by
\begin{equation}\label{eq:MCMCRayleighRest}
\lamMin{H}{A}:=\inf_{f\in H}\frac{\left\langle Af,f\right\rangle }{\left|f\right|^{2}}\text{ and }\lamMax{H}{A}:=\sup_{f\in H}\frac{\left\langle Af,f\right\rangle }{\left|f\right|^{2}}.
\end{equation}
The $L^2_\mu$-spectral gap is then given by the difference between 1 and the spectral radius of  the operator $P$ restricted to the orthogonal complement of the space of constant functions which we denote by $L_{0}^{2}(\mu)$. More precisely, we make the following definition.\\

\begin{definition}
\label{def:L2gap}($L_{\mu}^{2}$-spectral gap) A Markov kernel
${\MHK}$ with invariant measure $\mu$ has an  $L_{\mu}^{2}$-spectral
gap $1-\beta$ if 
\begin{equation}
\beta=\sup_{f\in L_{\mu}^{2}}\frac{\left\Vert P\left(f-\mu(f\right))\right\Vert _{2}}{\left\Vert f-\mu(f)\right\Vert _{2}}=\sup_{f\in L_{0}^{2}(\mu)}\frac{\left\Vert Pf\right\Vert _{2}}{\left\Vert f\right\Vert _{2}}=max \left( -\lamMin{\LOmu}{P},\lamMax{\LOmu}{P}\right) <1.
\end{equation}
\end{definition}
The existence of an $L^2_\mu$-spectral gap therefore requires a lower and an upper bound on the spectrum of $\lambda(P)$, thus on $\lambda_{min}^H(P)$ and $\lambda_{max}^H(P)$ respectively. The problem of obtaining a lower bound can be circumvented by considering 
\begin{equation}\label{eq:lazy}
\tilde{P}_r=rI+(1-r)P  \text{ for } 0 \leq r \leq \frac{1}{2}
\end{equation}
because  $\lambda_{min}^H (\tilde{P}_r) \ge  2r-1$.
More precisely, the corresponding Markov process can then be realised using the following steps:
\begin{itemize}
\item the Markov chain does not make a transition, with probability $r$.
\item the Markov chain makes an independent transition according to $P$ with probability $1-r$.
\end{itemize}
This process is a so-called lazy chain associated with $P$ and goes at least back to \cite{lovasz1998mixing}. It is straightforward to see that 
\[
	\sigma(\tilde{P}_r)\subseteq \left[ r+(1-r)\lamMin{\LOmu}{P}, r+(1-r)\lamMax{\LOmu}{P} \right].
\]
Thus, $\tilde{P}_r$ has a spectral gap if $\lamMax{\LOmu}{P} <1$. Its size can be optimised by choosing  the acceptance probability dependent on $\lamMax{\LOmu}{P} $.  We will refer to 
\begin{equation}\label{eq:upperSpectralGap}
 1-\lamMax{L^2(\posterior)}{P}
\end{equation}
as the upper $L^2_\posterior$-spectral gap.\\ 
\begin{main}
The main result of this article is stated in Theorem \ref{thm:main}. It concerns the spectrum of  the Metropolis-Hastings kernel arising from a proposal kernel $Q$ that is reversible with respect to $\prior$. If $\lamMax{\LOmuO}{Q}<1$ and $L$ in Equation (\ref{eq:target}) is bounded above and away from $0$, then also $\lamMax{\LOmu}{P}<1$. In particular, this implies a lower bound on the  $L^2_\mu$-spectral gaps of the lazy chains $\tilde{P}_r$ associated with $P$.
\end{main}
\vspace{0.5cm}

Our main result stated above is proved by expressing the $L^2_\mu$-spectral gap in terms of the associated Dirichlet form in Section \ref{sec:main}.  The proof is similar to that of the comparison theorem  in \cite{diaconis1993comparision}.  Our strategy for Bayesian inverse problems will be to design proposals that are reversible with respect to the prior and that have an  $L^2_{\mu_0}$-spectral gap. The Metropolis-Hastings algorithm will then perform an accept-reject step  according to the likelihood in order to produce samples from the posterior.

It is also worth mentioning that our main result as stated in Theorem \ref{thm:main} should be viewed in context of our recent results presented in \cite{hairer2011spectral} which are demonstrating that the $L^2$-spectral gap of the preconditioned Crank-Nicolson (pCN) algorithm with respect to the Gaussian reference measure is preserved for the corresponding Metropolis-Hastings kernel.  In the same article we used the Ornstein-Uhlenbeck proposal and assumed that $\densityTarRef$ in Equation (\ref{eq:target}) is log-Lipschitz.  However, no global bounds on $\densityTarRef$ were needed in order to prove the preservation of the  $L^2_\mu$-spectral gap. In this way the main result here can be viewed as an extension to a  much larger class of proposals and reference measures under partially stronger assumptions.  We also would like to mention that a related result has been proved for the Gibbs sampler applied to perturbations of Gaussian measures in \cite{1996AmitGibbs}.  However, it is not clear how it could be generalised to arbitrary reference measures.

\subsection{Non-Gaussian Priors, Uniform Series Priors and RRWM algorithms}\label{sec:overviewNonGaussian} 
We briefly summarise Besov priors and uniform series priors and explain how  Theorem \ref{thm:main} can be used in order to construct efficient proposals for posteriors arising from these kind of priors. In Section \ref{sec:Application} this programme is carried out in detail for the uniform series prior. 

As described in  \cite{stuartchinanotes}, many commonly used priors on function spaces can be written as a series expansion of the form 
\[
\prior=\mathcal{L}(u)=\mathcal{L}\left( \phi_0+\sum_{j=1}^\infty \gamma _j \gamma_j u_j \phi_j \right)
\]
where $\phi$ are elements of a possibly infinite Banach space  $X$, $\gamma_j$ are deterministic constants and $u_j$ are i.i.d. random variables. In case of Gaussian measures this fact follows from the Karhunen-Loeve expansion \cite{gaussianMeasureas,Adler:2007fk} where $u_j$ are i.i.d. $\mathcal{N}(0,1)$. For Besov priors, the $u_j$ are i.i.d. from the probability distribution with density proportional to 
\begin{equation}
density \propto \exp\left(-\frac{1}{2}|x|^q\right) \label{eq:Besov}
\end{equation}
where $q\ge 1$. For uniform series priors, the $u_j$ are i.i.d. from $\mathcal{U}(-1,1)$, the uniform distribution on $[-1,1]$. All these cases are of the form $u_j \overset{\text{i.i.d.}}{\sim} \nu $ for some probability measure $\nu$ on $\mathbb{R}$. A Markov kernel that has an $L^2_\nu$-spectral gap can be used compenentwise. By the tensorisation property of $L^2$-spectral gaps \cite{bakry2006functional,guionnet1801lectures}, the resulting Markov kernel has a spectral gap of the same size. 

There is a continuum of different choices but our intuition is that a proposal $Q(x,dy)$ for peaked measures should have a lot of mass close to $x$. Usually, we consider a collection of proposal kernels parametrised by a step size parameter $\epsilon$ where a smaller $\epsilon$ corresponds to the case that more mass is distributed closer to the current value. For smooth target densities, a small $\epsilon$ leads to a large acceptance rate because the ratio in Equation (\ref{eq:accG})  is then close to 1. However, if the step size is too small, the state space is not explored quickly. The trade off is common for Metropolis-Hastings algorithms, a recommendation for the step size can usually be obtained using diffusion limits, for details consider \cite{2012MattinglyDLimit} and references therein.

 For the uniform distribution such a proposal can be realised using a random walk with symmetric proposal 
\begin{align}
	Q_{\text{RW}}(x,dy)&=q(x-y)dy \nonumber \\
	Q_{\text{RW}}(x,dy)&=\mathcal{L}(x+\xi),\;\text{ where }\xi\sim \tilde{q} \nonumber
\end{align}
and repeatedly reflecting $y$ at the boundaries $-1$ and $1$. The reflection can be represented according to the following function
\begin{eqnarray*}
R(x):=\begin{cases}
y & y\le1\\
2-y & 1<y<3\\
-4+y & 3\le y\le4
\end{cases},\text{ where }y=x\text{ mod }\,4.
\end{eqnarray*}
In this way, we can write the proposal kernel as 
\begin{equation}
Q_\epsilon^{\text{RRWM}}(x,dy)=\mathcal{L}\left(R(x+ \epsilon \xi )\right) \label{eq:urmhPropLaw}
\end{equation}
 where $\xi\sim \tilde{q}$ and by introducing a step size parameter $\epsilon$.
We call the Metropolis-Hastings algorithms based on tensorisations of these proposals Reflection Random Walk Me\-tro\-po\-lis (RRWM) algorithms. These algorithms are revisited in Section \ref{sec:SGuniformSeries}. In particular, we consider the dependence of the $L^2_{\mathcal{U}(-1,1)}$-spectral gap of $Q_\epsilon$ on $\epsilon$ for $\tilde{q}=\mathcal{U}(-1,1)$ and $\tilde{q}=\mathcal{N}(0,1).$ Moreover,  in Section \ref{sec:num} we provide numerical evidence that the RRWM and the IS algorithms are robust with respect to an increase in dimension. These simulations also show that the RRWM algorithm is a substantial improvement over the IS algorithm especially for concentrated measures.

\begin{rem}\label{rem:besov}
For the Besov prior one possible choice is to use the transition kernel of a Metropolis-Hastings algorithm with target measure given by Equation (\ref{eq:Besov}) as building block for the proposal. Under mild conditions, the resulting Markov chains  are geometrically ergodic  
 \cite{Mengersen1996ConvergenceMCMC,Jarner2000341} and have therefore an $L^2_{\nu}$-spectral gap \cite{Roberts2001GeomL2}. In particular, by Theorem 3.2 in  \cite{Mengersen1996ConvergenceMCMC} this is true for the RWM algorithm and arbitrary $q\ge1$.
 \end{rem}

\section{Review of Bayesian Inverse Problems and Me\-tro\-po\-lis-Hastings Algorithms}\label{sec:review}
This section is devoted to giving a brief summary on the relevant material on Bayesian inverse problems and  to giving an introduction on Metropolis-Hastings algorithms on general state spaces.
For more details, we refer the reader to   \cite{MR2652785,stuartchinanotes} and \cite{samplingFirstInfiniteDimensional,brooks2011handbook} respectively.
The main idea of the Bayesian approach is to treat the parameters, the output of the mathematical model and the data as jointly distributed random variables. The randomness of the parameters is introduced artificially to subjectively model the uncertainty based on the a priori knowledge. The distribution of the parameters is called the prior. In the Bayesian framework the conditional probability distribution of the parameters  given the noisy data is called the  posterior. It is an update of the prior using the data and can be viewed as the solution to the inverse problem because it describes the a posteriori uncertainty about the parameters. The posterior is a very important tool because it can be used to
\begin{itemize}
\item  obtain point estimates for the unknown in an inverse problem such as the posterior mean or the MAP estimator which can be related to the Tikhonov regularisation, see \cite{2013dashtiMap};
\item quantify the uncertainty through the posterior variance or the posterior probability of a set in the parameter space.
\end{itemize}
We concentrate on the latter and note that both quantities can be written as posterior expectations. Metropolis-Hastings algorithms are used for that purpose and are reviewed in Section \ref{sec:MCMCGeneral}.

\subsection{Bayesian Inverse Problems}\label{sec:BIP}
We consider a general inverse problem for which the data  is generated by\begin{eqnarray*}
\data & = & \mathcal{G}(\input)+\noise \in Y.
\end{eqnarray*}
Here $\noise$ is the observational noise, $\input\in X$ is the input of the mathematical model, for example the initial condition or  coefficients  for a PDE, and $\mathcal{G}$ is the forward operator, a mapping between the Hilbert spaces $X$ and $Y$. In this setting, the inverse problem is concerned with the reconstruction of the input $\input$  to the model $\mathcal{G}$ given its noisy output, the data $\data$.  
In the Bayesian framework, this is approached by placing a prior probability measure $\prior$   on $\input$ containing all the
a priori information. If, in addition, the forward operator
$\mathcal{G}$ and the distribution of $\noise$ is given, then $\input$ and $\data$ can be treated as jointly varying random variables. Under mild assumptions, there exists a conditional probability measure on $\input$ which is called the posterior, an update of the prior using the data. In contrast to the minimiser of a least squares functional, the posterior is continuous in the data with respect to the total variation and the Hellinger distance. The posterior is also continuous with respect to approximations of the forward model. For the precise statements of these results we refer the reader to the surveys in  \cite{MR2652785} and \cite{stuartchinanotes}. Due to the latter result, it is possible to bound the difference between expectations calculated with respect to the posterior associated with the infinite dimensional and the discretised forward model. In Sections \ref{sec:MCMCGeneral} and \ref{sec:spectral}, we explain how the Metropolis-Hastings algorithm can be used  to approximate expectations  with respect to the posterior associated with the discretised forward model and how the resulting Monte-Carlo error can be bounded.
In order to use a Metropolis-Hastings algorithm, we specify the posterior more explicitly. For finite dimensional distributions given as probability densities Bayes' rule  \cite{1994BernardoBayesianBook} yields 
\begin{equation}
\text{posterior} \propto \text{likelihood} \times \text{prior} \label{eq:brule}.
\end{equation}
 We consider a generalisation of Bayes' rule to infinite dimensions. In this article, we only consider finite dimensional data, that is $Y=\mathbb{R}^N$. However, the results in  \cite{MR2652785} and \cite{stuartchinanotes} allow the data to be infinite dimensional as well. In the case of finite dimensional data, where the observational noise has a Lebesgue density $\rho$, the Bayesian framework can be summarised as follows

\begin{align}\label{eq:bayesInv}\begin{aligned}
\text{Prior}\qquad & \input\sim\mu_{0}\\
\text{Noise}\qquad & \noise\:\text{with pdf }\rho(\eta)\\
\text{Likelihood}\qquad & y|\input\:\text{ r.v. with pdf }\rho(y-\mathcal{G}(\input))\\
 & \densityTarRef(\input)=\rho(y-\mathcal{G}(\input))\\
\text{\text{Posterior}}\qquad & \frac{d\mu^{y}}{d\mu_{0}}(\input)\propto\densityTarRef(\input).
\end{aligned}\end{align}

Subsequently, we drop the dependence on the data $y$ and hope that this does not cause any confusion for the reader. The important point to note here is that  the Equations (\ref{eq:brule}) and (\ref{eq:bayesInv}) are of the same form as the general target measure for the Metropolis-Hastings algorithm in Equation (\ref{eq:target}) which will be reviewed in the subsequent section.

\subsection{The Metropolis-Hastings Algorithm on General State Spaces}

\label{sec:MCMCGeneral}

The common idea  of MCMC algorithms is to create a Markov
chain with a prescribed invariant measure, called the target measure $\mu$. Samples of this Markov chain  under (mild) conditions satisfy a law of large numbers and can thus be used to approximate expectations with respect to the target measure.  Under stronger conditions it is possible to control the resulting random error using a central limit theorem (CLT) or to establish bounds on the mean square error. We follow \cite{samplingFirstInfiniteDimensional} in introducing Metropolis-Hastings algorithms on general state spaces as this is needed for Bayesian inverse problems.

The idea of the Metropolis-Hastings kernel is to add an independent accept-reject step to a proposal Markov kernel
$\proposal(x,dy)$ in order to produce a Markov kernel $P(x,dy)$ with $\posterior$ as an invariant measure, that is
\[
\posterior P = \int_{X}\posterior(dx)P(x,dy)=\posterior(dy).
\]
Subsequently, we discuss a choice of the acceptance probability such that $\posterior$ is invariant for $P$. Thereafter we consider  the reversibility of both the proposal and the Metropolis-Hastings kernel.  This property is important because it yields error bounds on the sample average in combination with an   $L^{2}$-spectral gap (c.f. Section \ref{sec:spectral}). We close this section by reviewing convergence results for Metropolis-Hastings algorithms. 

The Metropolis-Hastings algorithm accepts a  move from $x$
to $y$ proposed by the kernel $Q(x,dy)$ with acceptance probability $\alpha(x,y)$. Thus, the algorithm takes the following form:\\ 

\noindent \textbf{Algorithm}
Initialise $X_{0}$. For $i\text{=0,\ensuremath{\dots},n}$ do:

\noindent Generate $Y\sim Q(X_{i},\cdot)$, $U\sim\mathcal{U}(0,1)$ independently and set
\[
X_{i+1}=\begin{cases}
Y & \text{if  }\alpha(X_{i},Y)>U\\
X_{i} & \text{otherwise}
\end{cases}.
\]

The transition kernel  $P(x,dy)$ associated with the Metropolis-Hastings algorithm can be written as 
\[
P(x,dy)=\alpha(x,y)Q(x,dy)+\delta_{x}(dy)\left(1-\int_{E}Q(x,dy)\alpha(x,y)\right).
\]

If the Radon-Nikodym derivative $\frac{d\posterior(dy)Q(y,dx)}{d\posterior(dx)Q(x,dy)}$
exists, then $\posterior$ is invariant for  $P$ for the choice 
\begin{equation}
\alpha(x,y):  =  \min\left(1,\frac{d\posterior(dy)Q(y,dx)}{d\posterior(dx)Q(x,dy)}\right). \label{eq:accG}
\end{equation}

In fact,  $\posterior$  is not only invariant for the Metropolis-Hastings kernel $P$ but the kernel $P$ is also reversible with respect to $\posterior$   \cite{samplingFirstInfiniteDimensional} as defined subsequently.\\

\begin{definition}
A Markov kernel $P$ is reversible with respect to a measure $\posterior$
if 
\[
\posterior(dx)P(x,dy)=\posterior(dy)P(y,dx).
\]
\end{definition}
If the proposal $Q$ is reversible with respect to a reference measure $\prior$ such that $\frac{d\posterior}{d\prior}\propto L$, then Equation (\ref{eq:accG}) reduces  to 
\begin{equation}\label{eq:accRev}
\alpha(x,y)=  \min\left(1,\frac{d\;CL(y)\prior(dy)Q(y,dx)}{d\; CL(x)\prior(dx)Q(x,dy)}\right)=\frac{\densityTarRef(y)}{\densityTarRef(x)} \wedge 1.
\end{equation}

The reference measure $\prior$ does not have to be a probability measure. In fact, in finite dimensions $\prior$ is often chosen to be the Lebesgue measure. The Markov kernel  associated with a symmetric random walks preserve the Lebesgue measure and therefore give rise to the simple acceptance ratio given above. The Random Walk Metropolis (RWM) algorithm is the well-known special case, if $Q(x,dy)=\mathcal{N}(x,C)(dy)$ for some positive definite covariance matrix $C$.

The problem in designing (efficient) proposals on function spaces is that the Radon-Nikodym derivative in Equation (\ref{eq:accG}) is often not well-defined. This follows from different almost sure properties  of $\mu(dx)$ and $\int Q(y,dx) d\mu(y)$ such as quadratic variation or regularity properties. The simplest proposal  which preserves these properties is to pick $\prior$ with the same almost sure properties and to use the proposal kernel
$$Q(x,dy)=\prior (dy).$$
 The resulting algorithm is called independence  sampler (IS) because the proposal does not depend on the current state $x$.
For Bayesian inverse problems it is natural to design proposals that are reversible for the prior because this preserves the almost sure properties and leads to a simple acceptance rule only involving the likelihood  (see also Equation (\ref{eq:accRev})). 

In general, Metropolis-Hastings algorithms are run in order
to approximate
\begin{equation}
\int\posterior(dx)f(x)\quad \text{ by }\quad S_{n,n_{0}}(f)=\frac{1}{n}\sum_{i=n_{0}}^{n_{0}+n}f(X_{i})\label{eq:burnin}
\end{equation}
where  $n_0$ is the burn-in corresponding to throwing away the first $n_0$ samples in order to reduce the bias. The resulting error takes the form
\[
  e_{n,n_0}(f)=\mu(f)-S_{n,n_{0}}(f).
\]
The complexity of Metropolis-Hastings algorithms can be quantified as 
\[
\text{number of necessary steps}\times\text{cost of one step.}
\]
The cost of one step is usually easy to quantify and depends on the problem at hand. The number of necessary steps depends on the prescribed error level (for example fixed width (asymptotic) confidence interval see \cite{2006JonesFixedWidth}, \cite{2011KrysConfidence} and \cite{explicitbdd})  and the convergence properties of the Markov chain. 

Geometric ergodicity is the most popular type of convergence used in the literature. An excellent review of this is given in \cite{roberts2004general}.  However, this approach is not well-adapted  to Bayesian inverse problems since the $\psi$-irreducibility property often fails to hold for the algorithm on a function space. By
$\psi$-irreducible we mean the existence of a positive measure $\psi$ such that 
\[
\psi(A)>0 \Rightarrow P(x,A)>0 \quad \forall x.
\]
This property often fails for infinite dimensional problems because the transition probabilities tend to be mutually singular for different starting points (this is even the case for a Gaussian random walk).  One exception is the IS algorithm  \cite{2013HoangComplexityMCMC}. Therefore the notion of $L^2_\mu$-spectral gaps, as introduced in Section \ref{sec:summaryMain}, is much better adapted to this situation and is used in the sequel.

Having introduced Bayesian inverse problems and Metropolis-Hastings algorithms on general state spaces, we are now in the position to formulate and prove the main result of this article.

\section{$L^2$-Spectral Gaps for Metropolis-Hastings algorithms}\label{sec:spectral}
Much of the theory on Metropolis-Hastings algorithms is aimed at establishing (asymptotic) bounds on the Monte Carlo error defined by
$$e_{n,n_0}(f):=\mu(f)-S_{n,n_{0}}(f)$$
where $S_{n,n_{0}}(f)$ is given by Equation (\ref{eq:burnin}). In this section, we survey the appropriate theorems from the literature that allow us to bound this error in terms of an  $L^2$-spectral gap justifying the importance of our main theorem. We start by introducing Dirichlet forms which relate to the second largest eigenvalue $\lamMax{L}{\LOmu}{P}$ of $P$ and therefore allow us to bound $\lamMax{L^2_0(\posterior)}{P}$. In this way we obtain a bound on the $L^2$-spectral gap of a lazy version of the Metropolis-Hastings chain in terms of 
$\lamMax{L^2_\prior}{Q}$ for the corresponding proposal chain and the bound on $L$.

\subsection{The Implications of an $L^{2}$-Spectral Gap}\label{sec:impl}
The two main implications of an $L^2_\target$-spectral gap are a CLT for $S_{n,n_{0}}(f)$ providing an asymptotic bound on the error of size $\mathcal{O} (\frac{1}{\sqrt{n}})$ and a non-asymptotic bound on the mean square error. The latter yields non-asymptotic confidence intervals using Chebyshev's inequality. 

In the  following we present the precise statement of the  CLT due to Kipnis and Varadhan  \cite{kipnis1986central}. The following version is taken from   \cite{Latuszynskiclt}.\\

\begin{proposition}(Kipnis-Varadhan)
\label{kipnisvara-clt} Consider an ergodic Markov chain with transition
operator ${\MHK}$ which is reversible with respect to a probability measure
$\posterior$ and which has an $L_{\mu}^{2}$-spectral gap $1-\beta$. For $f\in L^{2}$ we define 
\[
\sigma_{f,P}^{2}=\left\langle \frac{1+P}{1-P}f,f\right\rangle .
\]
 Then for $X_{0}\sim\mu$ the expression $\sqrt{n}(S_{n}-\mu(f))$
converges weakly to $\mathcal{N}(0,\sigma_{f,P}^{2})$. Moreover,
the following inequality holds
\[
\sigma_{f,P}^{2}\leq \frac{2\mu((f^{2}-\mu(f)^2))}{(1-\beta)}<\infty.
\]
\end{proposition}
The following non-asymptotic bounds on the mean square error are due to Rudolf \cite{explicitbdd}. \\

\begin{proposition}\label{prop:rudolf}(Rudolf)
Suppose that we have a Markov chain with Markov operator ${\MHK}$
having an $L_{\mu}^{2}$-spectral gap $1-\beta$. For $p\in(2,\infty]$
let $n_{0}(p)$ be chosen such that
\begin{equation}
n_{0}(p)\ge\frac{1}{\log\left(\beta^{-1}\right)}\begin{cases}
\frac{p}{2(p-2)}\log\left(\frac{32p}{p-2}\right)\left\Vert \frac{d\nu}{d\mu}-1\right\Vert _{L^{\frac{p}{p-2}}(\mu)} & p\in(2,4)\\
\log(64)\left\Vert \frac{d\nu}{d\mu}-1\right\Vert _{L^{\frac{p}{p-2}}(\mu)} & p\in[4,\infty].
\end{cases}\label{eq:burninsteps}
\end{equation}
Then for $S_{n,n_0}$ as in Equation (\ref{eq:burninsteps}) and $f\in L^2_\posterior$
\begin{equation}\label{eq:rudolf}
\underset{\left\Vert f\right\Vert _{2}\leq1}{\sup}\mathbb{E}\left[\left(\posterior(f))-\frac{1}{n}\sum_{i=n_{0}(p)}^{n_{0}(p)+n}f(X_{i})\right)^{2}\right]\leq\frac{2}{n(1-\beta)}+\frac{2}{n^{2}(1-\beta)^{2}}.
\end{equation}
\end{proposition}
\begin{rem}
The burn in $n_0(p)$ is chosen in terms of an a priori bound on $\frac{d\nu}{d\mu}-1$ in $L^{\frac{p}{p-2}}(\mu)$. The bound in Equation (\ref{eq:rudolf}) does not involve this quantity. For details, we refer the reader to \cite{explicitbdd}.
\end{rem}\\

If a Metropolis-Hastings algorithm has an  $L^2_\posterior$-spectral gap, then the two results above can be used to derive asymptotic and non-asymptotic confidence intervals and levels for  the Monte-Carlo error $e_{n,n_0}(f)$. The CLT only provides asymptotic confidence intervals. In contrast,  bounds on the MSE imply non-asymptotic confidence intervals using Chebyshev's inequality. Moreover, the size of the confidence intervals can be shrunk using the 'median trick' which estimates $\mathbb{E}f$  through the median of multiple shorter runs leading to exponential tight bounds. This trick has been developed for MCMC algorithms in \cite{niemiro2009MEDIAN}. Another good reference to mention is given by \cite{2011KrysConfidence}.

\subsection{Characterisation of the second largest eigenvalue of $P$ on $L^2(\mu)$}

The second largest eigenvalue $P$ on $L^2(\mu)$ is given by the largest eigenvalue  $\lamMax{\LOmu}{P}$ of $P$ on $L_{0}^{2}(\mu)$ can be obtained from the smallest eigenvalue $\lambda_{\text{min}}^{L_{0}^{2}(\mu)}(I-P)$
of $I-P$ on $L_{0}^{2}(\mu)$.

This can be characterised as follows
\begin{equation}
1-\lambda_{\text{max}}^{L_{0}^{2}(\mu)}(P)=\inf_{f\in L_{0}^{2}(\mu)}\frac{\left\langle (I-P)f,f\right\rangle }{\left|f\right|^{2}}=\inf_{f\in L^{2}(\mu)}\frac{\left\langle (I-P)\Pi f,\Pi f\right\rangle }{\left|\Pi f\right|^{2}}\label{eq:upperSG}
\end{equation}
where $\Pi:L_{0}^{2}(\mu)\rightarrow L^{2}(\mu)$ is the orthogonal
projection onto $L_{0}^{2}(\mu)$ given by
\[
\Pi f=f-\mu(f).
\]
The denominator can be rewritten as 
\begin{eqnarray}
\left|\Pi f\right|^{2} & = & Var_{\mu}(f)=\int\left(f-\mu(f)\right)^{2}d\mu\nonumber \\
 & = & \int f^{2}d\mu-\mu(f)^{2}=\frac{1}{2}\int\mu(dx)\mu(dy)\left(f(x)-f(y)\right)^{2}.\label{eq:variance}
\end{eqnarray}
The nominator in Equation (\ref{eq:upperSG}) can be rewritten as
\begin{eqnarray*}
\left\langle (I-P)(f-\mu(f)),f-\mu(f)\right\rangle  & = & \left\langle (I-P)f,f-\mu(f)\right\rangle =\left\langle (I-P)f,f\right\rangle \\
 & = & \int\mu(dx)P(x,dy)\left(f(x)^{2}-f(x)f(y)\right)dy\\
 & = & \frac{1}{2}\int\mu(dx)P(x,dy)\left(f(x)-f(y)\right)^{2}dy=:\mathcal{E}_{\mu}^{P}(f,f).
\end{eqnarray*}
The bilinear form $\mathcal{E}(f,f)$ is the Dirichlet form associated
with the Markov chain given through the transition kernel $P$. There
is a large literature on studying Markov processes through their Dirichlet
form. We refer the reader to \cite{Schmuland1999Byron} for a short
survey on time-continuous Markov processes, to \cite{Fukushima2011DirichletForms}
for a full account of the theory and to \cite{levin2009markov} for
a review on discrete Markov chains.  In the subsequent derivation of our main theorem, we only use the characterisation of the $L^2_\posterior$-spectral gap
\begin{equation}
1-\lambda_{\text{max}}^{L_{0}^{2}(\mu)}=\inf_{f\in L^{2}(\mu)}\frac{\mathcal{E}_{\mu}^{P}(f,f)}{\text{Var}(f)} \label{eq:SGdirichlet}
\end{equation}
 that we have just derived.

\subsection{Main Result}\label{sec:main}

The following theorem provides an explicit lower bound on $\lambda_{\text{max}}^{L_{0}^{2}(\mu)}(P)$ of the  Metropolis-Hastings chain in terms of the eigenvalue $\lambda_{\text{max}}^{L_{0}^{2}(\mu_{0})}(Q)$ of the proposal chain and the  bounds on the density of the posterior with respect to the prior. It is close in spirit to the comparison theorem
for discrete Markov chains obtained in \cite{diaconis1993comparision}.\\

\begin{theorem}\label{thm:main}
Suppose that the proposal kernel $Q$ satisfies a lower bound on the
upper $L_{\mu_{0}}^{2}$-spectral gap $1-\lambda_{\text{max}}^{L_{0}^{2}(\mu)}(Q)>0$
and assume that the target measure takes the form
\[
\mu=\frac{L}{Z}\mu_{0}.
\]
Then the upper $L_{\mu}^{2}$-spectral gap satisfies 
\[
\left(1-\lambda_{\text{max}}^{L_{0}^{2}(\mu_{0})}(Q)\right)\frac{L^{\star3}}{L_{\star}^{3}}\ge1-\lambda_{\text{max}}^{L_{0}^{2}(\mu)}(P)\ge \frac{L_{\star}^{4}}{L^{\star4}}\left(1-\lambda_{\text{max}}^{L_{0}^{2}(\mu_{0})}(Q)\right)
\]
where $L_{\star}:=\inf L\leq L\leq\sup L=L^{\star}.$ In particular,
the lazy version $\tilde{P}_r$, given in Equation (\ref{eq:lazy}), has an $L_{\mu}^{2}$-spectral gap $1-\beta_{\text{lazy}}$
satisfying 
\[
1-\beta_{\text{lazy}}\ge \min \left(r+(1-r)(1+\lamMin{L^2_0 (\posterior))}{P}), r+(1-r) \frac{L_{\star}^{4}}{L^{\star4}}\left(1-\lambda_{\text{max}}^{L_{0}^{2}(\mu_{0})}(Q)\right)  \right).
\]
\end{theorem}
\begin{proof}
From Equation (\ref{eq:variance}) it follows that
\[
\frac{L_{\star}^{2}}{Z^{2}}\text{Var}_{\mu}(f)\leq\text{Var}_{\mu_{0}}(f)\leq\frac{L^{\star2}}{Z^{2}}\text{Var}_{\mu}(f).
\]
Similarly, we notice that
\begin{eqnarray*}
\mathcal{E}_{\mu}^{P}(f,f) & = & \frac{1}{2}\int\mu_{0}(dx)Q(x,dy)\frac{L}{Z}\alpha(x,y)\left(f(x)-f(y)\right)^{2}\\
 & \ge & \frac{L_{\star}}{Z}\alpha_{\star}\frac{1}{2}\int\mu_{0}(dx)Q(x,dy)\left(f(x)-f(y)\right)^{2}\\
 & \ge & \frac{L_{\star}^{2}}{ZL^{\star}}\left(1-\lambda_{\text{max}}^{L_{0}^{2}(\mu_{0})}(Q)\right)\text{Var}_{\mu_{0}}(f)\\
 & \ge & \frac{L_{\star}^{4}}{Z^{3}L^{\star}}\left(1-\lambda_{\text{max}}^{L_{0}^{2}(\mu_{0})}(Q)\right)\text{Var}_{\mu}(f)\\
 & \ge & \frac{L_{\star}^{4}}{L^{\star4}}\left(1-\lambda_{\text{max}}^{L_{0}^{2}(\mu_{0})}(Q)\right)\text{Var}_{\mu}(f).
\end{eqnarray*}
Thus, we can conclude that 
\[ 1-\lambda_{\text{max}}^{L_{0}^{2}(\mu)}(P)=\inf_{f\in L^{2}(\mu)}\frac{\mathcal{E}_{\mu}^{P}(f,f)}{\text{Var}(f)}\ge\frac{L_{\star}^{4}}{L^{\star4}}\left(1-\lambda_{\text{max}}^{L_{0}^{2}(\mu_{0})}(Q)\right). \]
The other inequality is obtained in the following way

\begin{eqnarray*}
\mathcal{E}_{\mu}^{Q}(f,f) & = & \frac{1}{2}\int\mu_{0}(dx)Q(x,dy)\frac{L}{Z}\alpha(x,y)\left(f(x)-f(y)\right)^{2}\\
 & \ge & \frac{L_{\star}}{Z}\frac{1}{2}\int\mu(dx)P(x,dy)\left(f(x)-f(y)\right)^{2}\\
 & \ge & \frac{L_{\star}}{Z}\left(1-\lambda_{\text{max}}^{L_{0}^{2}(\mu)}(P)\right)\text{Var}_{\mu}(f)\\
 & \ge & \frac{L_{\star}^{3}}{Z^{3}}\left(1-\lambda_{\text{max}}^{L_{0}^{2}(\mu)}(P)\right)\text{Var}_{\mu_{0}}(f).
\end{eqnarray*}
The result for the lazy chain follows from the discussion at the beginning of this section. \end{proof}\\

\begin{rem}
 Using a lazy version of a Markov chain results in a worse asymptotic performance. Corollary 1 from  \cite{Latuszynskiclt} states that the asymptotic variance of $S_n(f)$ for the Markov chain associated with $\tilde{P}_r$ is given by
\[
\sigma^2_{f,\tilde{P}_r}=\frac{1}{1-r}\sigma^2_{f,P}+\frac{r}{1-r}\sigma^2_f\ge \sigma^2_{f,P}.
\]
However, the proof of Proposition \ref{prop:rudolf} in \cite{explicitbdd} crucially requires a lower bound on the $L^2$-spectral gap to obtain a non-asymptotic bound on the mean square error.\\
\end{rem}

This result highlights the insight that the reference measure is crucial for designing efficient sampling algorithms on function spaces. A typical example would be the use of a Markov chain that has an $L^2_\prior$-spectral gap where $\prior$ is the prior of a Bayesian problem. If the likelihood is bounded, then the lazy version of the resulting Metropolis-Hastings algorithm with this chain as the proposal has an $L_\posterior^2$-spectral gap with $\posterior$ being the posterior. However, the result is not limited to this situation because $\prior$ and $\posterior$ can be arbitrary measures such that the density of $\posterior$ with respect to $\prior$ is bounded.\\

\begin{rem}
\label{rem:independenceSampler1}
For a fixed target measure  a larger $L^2_{\mu_0}$-spectral gap of $\proposal$ implies a larger lower bound on  the $L^2_\posterior$-spectral gap of $\MHK$. In particular the largest lower bound is achieved for the IS algorithm. It is important to note that this does not imply that this choice leads to the largest spectral gap for $\MHK$. In fact, the simulations in Section \ref{sec:num} as demonstrated in Figure \ref{fig:autocorr} and \ref{fig:autocorr2} suggest otherwise.\\
\end{rem}
\begin{rem}
\label{rem:independenceSampler2}
The results obtained in \cite{1996AmitGibbs} for the Gibbs sampler applied to a perturbation of a Gaussian measure suggest that the sharper inequalities
\[
\left(\frac{\densityTarRef_{\star}}{\densityTarRef^{\star}}\right){(1-\beta_{\text{prop}})}\leq1-\beta\leq \left(\frac{\densityTarRef^{\star}}{\densityTarRef_{\star}}\right){(1-\beta_{\text{prop}})}
\]
might hold. This seems to be an interesting question for further investigation. 

\end{rem}

\section{Application to an Elliptic Inverse Problem\label{sec:Application}}
The theoretical result of the previous section was motivated by studying the reconstruction of the diffusion coefficient $\input$ given by noisy observations of the pressure $\pressure$.  
In Section \ref{sec:forward} we set up the forward problem  and review the literature on  the resulting inverse problem  focusing on the Bayesian approach.  In Section \ref{sec:appl.bayesian} we specify the Bayesian inverse problem by setting up a uniform series prior and specify the noise to be additive and Gaussian. Moreover, we show  that the resulting posterior has a bounded density with respect to this prior. 
The remaining part of the section is devoted to constructing appropriate proposal kernels and proving a lower bound on their $L^2_\prior$-spectral gaps revisiting the ideas from Section \ref{sec:overviewNonGaussian}. Thus, our main theorem  implies a lower bound on the  $L^2_\posterior$- spectral gaps of the corresponding Metropolis-Hastings algorithms.

\subsection{The Underlying PDE and Well-Definedness of the Forward Model\label{sec:forward}}
The forward problem is based  on  the relation between the pressure $\pressure$ and the diffusion coefficient	$\input$  modelled  by the following elliptic PDE with Dirichlet boundary conditions
\begin{align}\label{eq:postConEllipticForward}
\begin{aligned}
\begin{cases} -\nabla\cdot(\input\nabla\pressure)&=g(x)\;\hspace{0.2cm}\text{in }D\\
\quad\pressure&=0\qquad\text{ on }\partial D \end{cases}
\end{aligned}
\end{align}
where $D$ is a bounded domain in $\mathbb{R}^{d}$ and $\pressure$ and $\input$ are scalar functions on $D$. We assume that $a^{\star}\ge a(x)\ge a_\star >0$ for  almost every $x\in D$.  The subset of $L^\infty(D)$-functions that satisfy this condition is denoted by
 \[L^\infty_{+}:=\left\{u\in L^\infty\Big\vert\ \underset{D}{\text{ess inf }} u>0\right\}.\]
If, additionally, $g$ is  in the Sobolev space $H^{-1}$, then  the solution operator $\pressure(x;\input):L^\infty_+\rightarrow H^1$,
 mapping to the unique weak solution of the PDE stated in Equation (\ref{eq:postConEllipticForward}), is well-defined (for details we refer the reader to \cite{stuartchinanotes}). We suppose that the forward operator $\mathcal{G}$, giving rise to the data, is based on the solution operator denoted by $\pressure(\cdot;\input)$ and the observation operator $\mathcal{O}$ as follows
 \begin{equation}
\mathcal{G}(\input)=\mathcal{O}\left(\pressure(\cdot;\input)\right) \label{eq:assForwardSplitting}
\end{equation}
Additionally, we suppose that it is equal to  $\mathcal{O}=\left (l_1,\dots,l_N\right)$ with  $l_i\in H^{-1}$.

The inverse problem associated with the above forward problem is well-known and it is particularly relevant in oil reservoir simulations and the modelling of groundwater flow, see for example  \cite{mclaughlin1996reassessment}. A survey on classical least squares approaches to this inverse problem can be found in \cite{1995KunischNumericalEIP} for which error estimates have been obtained recently in \cite{2010WangErrorEIPReg}.  A rigorous Bayesian formulation of this inverse problem with log-Gaussian priors and Besov priors is given in  \cite{UncertaintyElliptic} and \cite{Con3} respectively, both are reviewed in \cite{stuartchinanotes}. There is also an extensive literature in the uncertainty quantification community studying how uncertainty propagates through the forward model. This can be investigated by considering different realisations of the input.  This approach can be combined with the finite element  \cite{ghanem2003stochastic} and Galerkin methods \cite{babuska2004galerkin} used to approximate the underlying equation. For the elliptic inverse problem under consideration, this has been studied in \cite{2010SchwabEllipticUQ}. In fact, it can be more efficient to use generalised Polynomial Chaos (gPC) \cite{schwab2011sparseUQelliptic} instead of Monte Carlo methods. Recently, gPC methods have also been applied to the elliptic inverse problem considered in \cite{Con4,2013HoangComplexityMCMC}. Since gPC often suffers from a large constant and has only been developed for a few inverse problems, it is important to construct efficient samplers tailored to the prior and likelihood at hand. Moreover, we would like to mention that it is also possible to speed up MCMC algorithms using the multi level approach. The expectation of interest is written as difference corresponding to a finer and finer discretisation so that more MCMC samples are used for coarser discretisations \cite{2013HoangComplexityMCMC}. 

\subsection{A Bayesian Approach\label{sec:appl.bayesian}}
We set up the Bayesian inverse problem by specifying the prior and the distribution of the observational noise. Morevoer, we derive a bound on the posterior density which allows us to use Theorem \ref{thm:main} at the end of Section \ref{sec:SGuniformSeries}.
  Following \cite{2013HoangComplexityMCMC} and \cite{Con4}, we choose a prior on the coefficients $(u_1,\dots,u_J)$  for $J\in \mathbb{N}\cup \{\infty\}$ giving rise to the diffusion coefficient
\begin{eqnarray}
\input(u)(x) & = & \bar{\input}(x)+\sum_{j\in\mathbb{J}}\gamma_{j}u_{j}\psi_{j}(x)\label{eq:priorExpansion}
\end{eqnarray}
\noindent where $\norm{\psi_i}_{L^\infty}=1$.  We suppose that $u_i \overset{\text{i.i.d.}}{\sim} \mathcal{U}(-1,1)$ which corresponds to a prior given by 
\[ \prior^J = \bigotimes_{j=1}^J  \mathcal{U}(-1,1.)\]
Additionally, we assume that the weights $\gamma_i$   are such that $\inf a(x)(x)$ is bounded away from $0$ uniformly in $J$  implying that the solution operator $p$ is well-defined for $\prior$-almost every $a(u)$. 
We would like to note that similar probability measures have been studied for the propagation of uncertainty in \cite{2010SchwabEllipticUQ}.

We suppose that the data is given by 
\[
\data  =  \mathcal{G}(\input(u))+\noise
\]
where $\noise \sim \mathcal{N}(0,\Gamma)$. The well-definedness of the corresponding posterior for $J\in \mathbb{N} \cup \{\infty\}$ has been proven in  \cite{Con4} and \cite{stuartchinanotes}. It takes the form
\[
	\frac{d\posterior}{d\prior}\propto \exp\left( -\frac{1}{2} \norm{\data -\mathcal{G}(a)}^2_\Gamma \right).
\]
We also know that 
\[
\left\Vert \mathcal{G}(a)\right\Vert _{\Gamma}\le\left\Vert \Gamma\right\Vert _{2}N\max_{i}\left\Vert l_{i}\right\Vert _{H^{-1}} \sup_{-1\le a_i \le 1}  \left\Vert p(a)\right\Vert _{H^{1}}\leq C\left\Vert \Gamma\right\Vert _{2}N\max_{i}\left\Vert l_{i}\right\Vert _{H^{-1}}a_{\star}^{-1}
\]
\noindent where $a_\star:=\underset{D}{\text{ess inf }}a$.
Note that $C$  depends on $N$ (see Equation (\ref{eq:assForwardSplitting})) but can be chosen uniformly in $J$. This gives rise to the following upper and lower bounds on the likelihood
\begin{align*}
L^{\star} & =1\\
L_{\star} & =\exp\left(-2C^{2}\left\Vert \Gamma^-1\right\Vert _{2}^{1}N^{2}\left(\max_{i}\left\Vert l_{i}\right\Vert _{H^{-1}}\right)^{2}a_{\star}^{-2}\right).
\end{align*}
These bounds are crucial because they allow us to use Theorem \ref{thm:main} in the next section.

\subsection{Spectral Gaps for the Prior and the Posterior}\label{sec:SGuniformSeries}

In order to apply our main result to the setting of this Bayesian inverse problem, we have to choose a proposal kernel $Q$ that is reversible and has an $L^{2}_{\prior}$-spectral gap with respect to  $\prior=\mathcal{U}{(-1,1)}^J$. In the following we work out the details of constructing a proposal for uniform series priors revisiting the ideas presented in Section \ref{sec:overviewNonGaussian}. As described in Remark \ref{rem:besov}, these ideas can also be generalised to Besov priors.

Given any kernel that has an $L^2_{\mathcal{U}{(-1,1)}}$-spectral gap we may apply  the tensorisation property of $L^2$-spectral gaps \cite{bakry2006functional,guionnet1801lectures} in order to conclude that  this application to each component yields  a kernel with the same size spectral gap for $\mathcal{U}{(-1,1)}^J$. Whereas we construct the one dimensional proposal distributions explicitly below, it is worth pointing out that it is possible to obtain an appropriate one-dimensional proposal using the Metropolis-Hastings kernel for $\mathcal{U}{(-1,1)}$ with a one-dimensional proposal distribution. Thus, the resulting Markov kernel is uniformly ergodic under mild assumptions \cite{roberts2004general} implying an $L^2_{\mathcal{U}{(-1,1)}}$-spectral gap \cite{Roberts1997Hybrid}.  Note that the resulting proposal on $[-1,1]^J$ can be accepted even if some of the one-dimensional Metropolis-Hastings algorithms have rejected their proposal.

Alternatively, a Markov kernel with $L^2_{\mathcal{U}(-1,1)}$-spectral gap can be obtained by considering  a random walk with symmetric proposal 
\begin{align}
	Q_{\text{RW}}(x,dy)&=q(x-y)dy \nonumber \\
	Q_{\text{RW}}(x,dy)&=\mathcal{L}(x+\xi),\;\text{ where }\xi\sim \tilde{q} \nonumber
\end{align}
and repeatedly reflecting $y$ at the boundaries $-1$ and $1$. The reflection can be represented according to the following function \begin{eqnarray*}
R(x)=\begin{cases}
y & y\le1\\
2-y & 1<y<3\\
-4+y & 3\le y\le4
\end{cases},\text{ where }y=x\text{ mod }\,4.
\end{eqnarray*}
We call the  Metropolis-Hastings algorithm based on tensorisations of this proposal Reflection Random Walk Me\-tro\-po\-lis (RRWM) algorithm. In this way we can write the proposal kernel as 
\begin{equation}
Q^{\text{RRWM}}(x,dy)=\mathcal{L}\left(R(x+ \xi )\right) \label{eq:urmhPropLaw}
\end{equation}
 where $\xi\sim \tilde{q}$. Its density with respect to the Lebesgue measure takes the form
\begin{equation}
	q^{\text{RRWM}}(x,dy)=\sum_{k\in\mathbb{Z}} \tilde{q}(x-y+4k)+\tilde{q}(x+y+4k+2).
\end{equation}
The  proposal kernel $Q_{\text{RRWM}}$ is reversible with respect to $\mathcal{U}(-1,1)$ because $$q^{ \text{RRWM}}(x,y)=q^{\text{RRWM}}(y,x).$$
In the following we consider the RRWM algorithm with uniform random walk ($\xi\sim\mathcal{U}(-\epsilon,\epsilon)$) and with standard random walk ($\xi \sim \mathcal{N}(0,\epsilon^2)$) which we call Reflection Uniform Random Walk Metropolis (RURWM) and Reflection Standard Random Walk Metropolis (RSRWM) algorithm, respectively. In contrast to the RSRWM algorithm, the density of the RURWM has a closed form  $q^{\text{RURWM}}_{\epsilon}$. For $\epsilon<1$ it is given by
\begin{equation}
q^{\text{RURWM}}_{\epsilon}(x,y)  \propto  \begin{cases}
1 & \text{if}-1\leq x,y\leq1,\,\left|x-y\right|\leq\epsilon,y>-x-2+\epsilon,y<-x+2-\epsilon\\
2 & \text{if}-1\leq x,y\leq1,\, y\le-x-2+\epsilon\text{ or }y\ge-x+2-\epsilon\\
0 & \text{otherwise}
\end{cases}\label{eq:reflection}.
\end{equation}
The following result shows that the lazy versions of the  RURWM algorithms have an $L^{2}$-spectral gap of order $\epsilon ^2$.\\

\begin{theorem}
\label{thm:URMHproposal} There is $c>0$ such that the $L_{\mathcal{U}(-1,1)}^{2}$-spectral
gap $1-\beta_{\epsilon}$ of \textup{$Q_{\epsilon}^{\text{RURWM}}$
for }$\epsilon\leq1$ satisfies
$1-\beta_{\epsilon}\ge c\epsilon^{2}.$\\
\end{theorem}
\begin{proof}
See Appendix \ref{sec:appendixProof}.
\end{proof}\\

In a similar manner it can also be shown that the proposal of the RSRWM algorithm has an $L^2_{\mu_0}$-spectral gap of order $\epsilon$. In particular, the lower bound on the transition density of the random walk can be obtained more easily because the $n$-step transition kernel has an explicit form.

The lower bound on the spectral gap of the resulting lazy versions of Metropolis-Hastings algorithms follows now from Theorem \ref{thm:main}.\\

\begin{corollary}\label{thm:application}
Let $Q$ be a Markov kernel that has an $L^2_{\mathcal{U}_{(-1,1)}}$-spectral gap $1-\beta_\text{prop}$, $J\in \mathbb{N} \cup \{\infty \}$ and $Q_{J}=\bigotimes_{j=1}^{J}Q(\input_{j},d\tilde{a}_{j})$. Then the lazy version of the Metropolis-Hastings transition kernel ${\MHK}_{J}$ for $\posterior_J$  
 with proposal $Q_{J}$ has an $L^{2}_{\posterior_J}$-spectral gap $1-\beta_{J}$
and there is a $J$-independent lower bound of the form
\[  1-\beta \ge  \frac{1}{2}\exp\left(-8C^{2}\left\Vert \Gamma^-1\right\Vert N^{2}\left(\max_{i}\left\Vert l_{i}\right\Vert _{H^{-1}}\right)^{2}a_{\star}^{-2}\right)(1-\beta_\text{prop})^2.\]
\end{corollary}
In this section, we have constructed the RRWM algorithm for the elliptic inverse problem with prior based on a series expansion with uniformly distributed coefficients. In the next section, we compare this algorithm to the IS and RWM algorithms using simulations.

\section{Numerical Comparison of Different MCMC Algorithms for a particular Elliptic Inverse Problem}\label{sec:num}

We apply the Random Walk Metropolis (RWM), the  Importance Sampling (IS) and the Reflection Random Walk Metropolis (RRWM) algorithms to the posterior arising from the elliptic inverse problem considered in Section \ref{sec:Application}. We use simulations to illustrate the following two aspects:
\begin{itemize}

\item On the one hand the acceptance probability of the standard RWM algorithm decreases quickly as the dimension of the state space increases. On the other hand, the relation between the step size and the acceptance probability of the IS and RRWM algorithms are not affected by the dimension.
\item The performance of  the IS algorithm is only affected up to a point by the dimension $J$ of the state space. However, it does not perform well for concentrated target measures. In contrast our simulations show that the choice of an appropriate step size for the RRWM algorithm leads to a good performance for the problem at hand. \end{itemize}

\noindent We first describe the implementation of the forward model, the choice of the prior and the implementation of the IS,  the RWM and the RRWM algorithms. Even-though our result only applies to the lazy version of the Metropolis-Hastings algorithm, we believe that this is artificial and present simulations for the non-lazy versions.

The remaining part of the section is then divided into presenting the dependence of the relationship  between step size and acceptance rate on the dimension as well as the decay of the autocorrelation. 

\subsection{The Setup}\label{sec:setup}
We consider the elliptic inverse problem as described in Section \ref{sec:Application} on the domain  $D=[0,1]$. In this case there is an explicit formula linking the pressure  $p$ and the diffusion coefficient $a$
which has been implemented using a trapezoidal rule. We choose the prior as in Equation (\ref{eq:priorExpansion}) on the coefficients $u_i$, that is
\[ \prior^J = \bigotimes_{j=0}^J  \mathcal{U}(-1,1).\]
These coefficients give rise to the diffusion coefficient
\begin{equation} \label{eq:simulations} \input(u)(x)  =  \bar{\input}(x)+\sum_{j=0}^J\gamma_{j}u_{j}\psi_{j}(x)\text{ where }u_{j}\overset{\text{i.i.d.}}{\sim}\mathcal{U}(-1,1). \end{equation}
For our simulations we set
\begin{eqnarray*}
    \bar{a}(x )&= &4.38.\\
\psi_{2j-1}(x) & = & \cos(2\pi jx),\,\;\gamma_{2j}=\frac{1}{j^{2}},\,K \ge  j\ge1\\
\psi_{2j}(x) & = & \sin(2\pi jx),\;\gamma_{2j-1}=\frac{1}{j^{2}},\, K \ge  j\ge1\\
\psi_{0}(x) & = & 1,\;\gamma_{0}=1
\end{eqnarray*}
where $K$ denotes the number of Fourier coefficients.
Note that the lower bound  $a(x)\ge 1$  is independent of $J=2K$. The data $\data$ corresponds to evaluations of the pressure $p$ on an evenly spaced grid. More precisely,
\[
\data=\opObs(\input^{\dagger})+\noise=(p(i \cdot \disObs) +\noise_i)_{i=0}^{\left\lfloor 1/\disObs\right\rfloor } 
\]
where $\noise\sim\mathcal{N}(0,\sigma^{2}I)$ and $\input^{\dagger}$ is a fixed draw from the prior. 

Subsequently, we consider the IS, the RWM, the RURWM and the RSRWM algorithms with the following proposal kernels
\begin{align*}
 Q^{\text{IS}}(x,dy) &= \prior(dy)\\
 Q_{\epsilon}^{\text{RWM}}(x,dy) &= \mathcal{N}\left(x,\epsilon I_{d\times d}\right)(dy)\\
Q_\epsilon^{\text{RURWM}}(x,dy)&=\otimes_{i=1}^{d}\mathcal{L}\left(R(x+\epsilon \xi) \right),\;\xi\sim \mathcal{U}(-1,1)\\
Q_\epsilon^{\text{RSRWM}}(x,dy)&=\otimes_{i=1}^{d}\mathcal{L}\left(R(x+\epsilon \xi) \right), \;\xi\sim \mathcal{N}(0,1).
\end{align*}

Note that the Metropolis-Hastings acceptance ratio, as described in Section \ref{sec:review}, implies that the RWM algorithm simply rejects any proposal outside the unit cube.

\subsection{Acceptance Probabilities for the RWM and the RRWM Algorithms}
In Figure \ref{fig:acc}, we have plotted the acceptance probability against the step size for the RWM, the RURWM and the RSRWM  algorithm for different choices of $K$. The target for both is the posterior arising from 33 artificially generated measurements with $\sigma=0.05$. 

The step size parameter  $\epsilon$ affects the performance of all three algorithms. On the one hand large step sizes are beneficial because the algorithm can explore the state space quicker whereas they lead to a small acceptance ratio (see Figure \ref{fig:acc}). On the other hand small step sizes lead to a high acceptance ratio but to highly correlated samples. The IS algorithm does not have a step size parameter and its average acceptance probability does not depend on the dimension. For this choice of parameters it is approximately $4.4\%$.

  Figure \ref{fig:acc} clearly illustrates that the acceptance probability of the RWM algorithm for a fixed step size deteriorates as the dimension increases. One reason for the decay of the acceptance probability of the RWM algorithm is that the probability of the proposal lying outside $[0,1]^d$ increases to $1$ as $d\rightarrow \infty$. Moreover, there is no visible impact of the dimension on the acceptance probability for the RURWM and the RSRWM algorithms.

\begin{figure}
\begin{center}
\subfloat[Acceptance rate vs. step size for the RWM algorithm]{\includegraphics[width=0.7\textwidth]{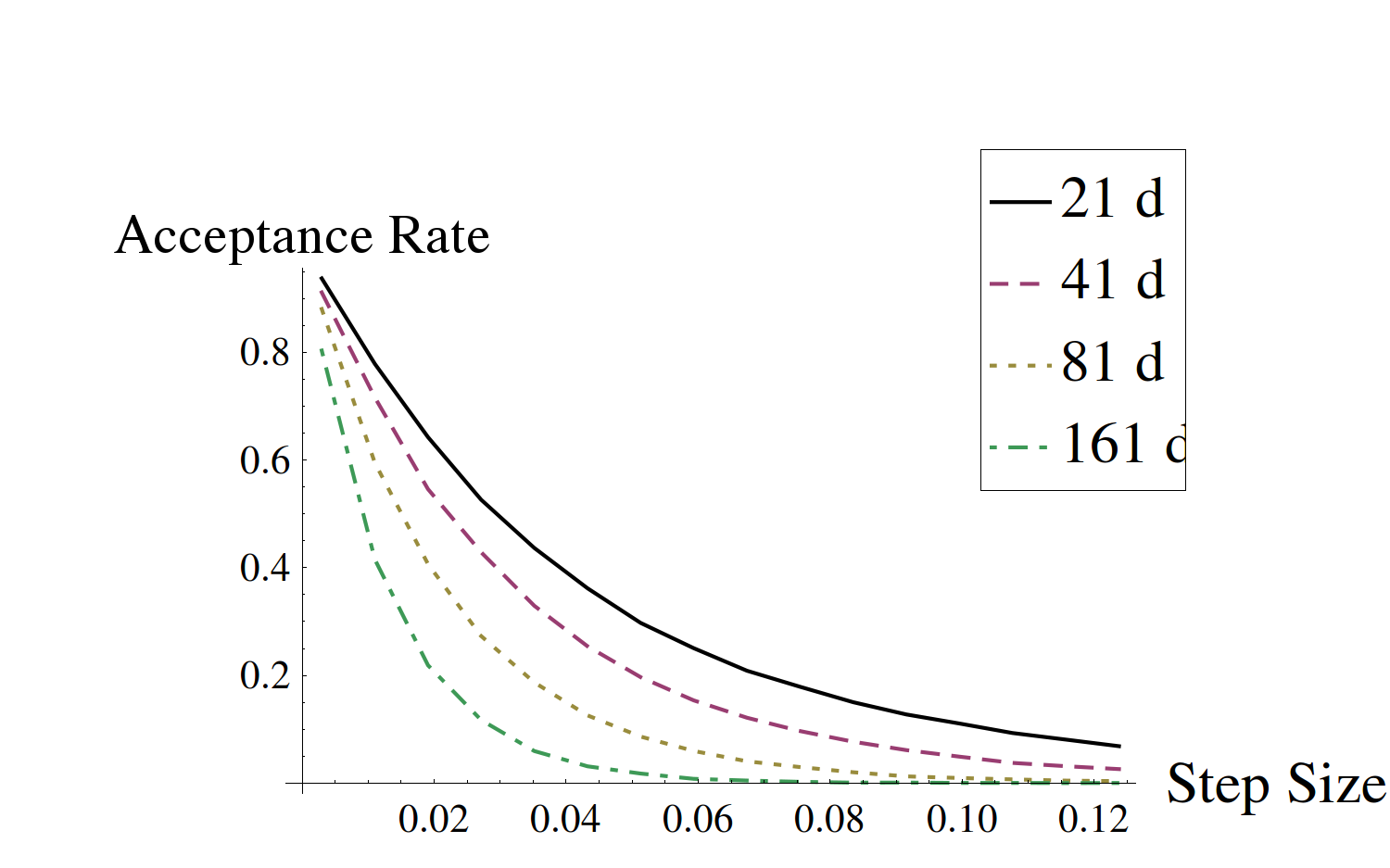}
}\\
\subfloat[Acceptance rate vs. step size for the RURWM algorithm ]{\includegraphics[width=0.7\textwidth]{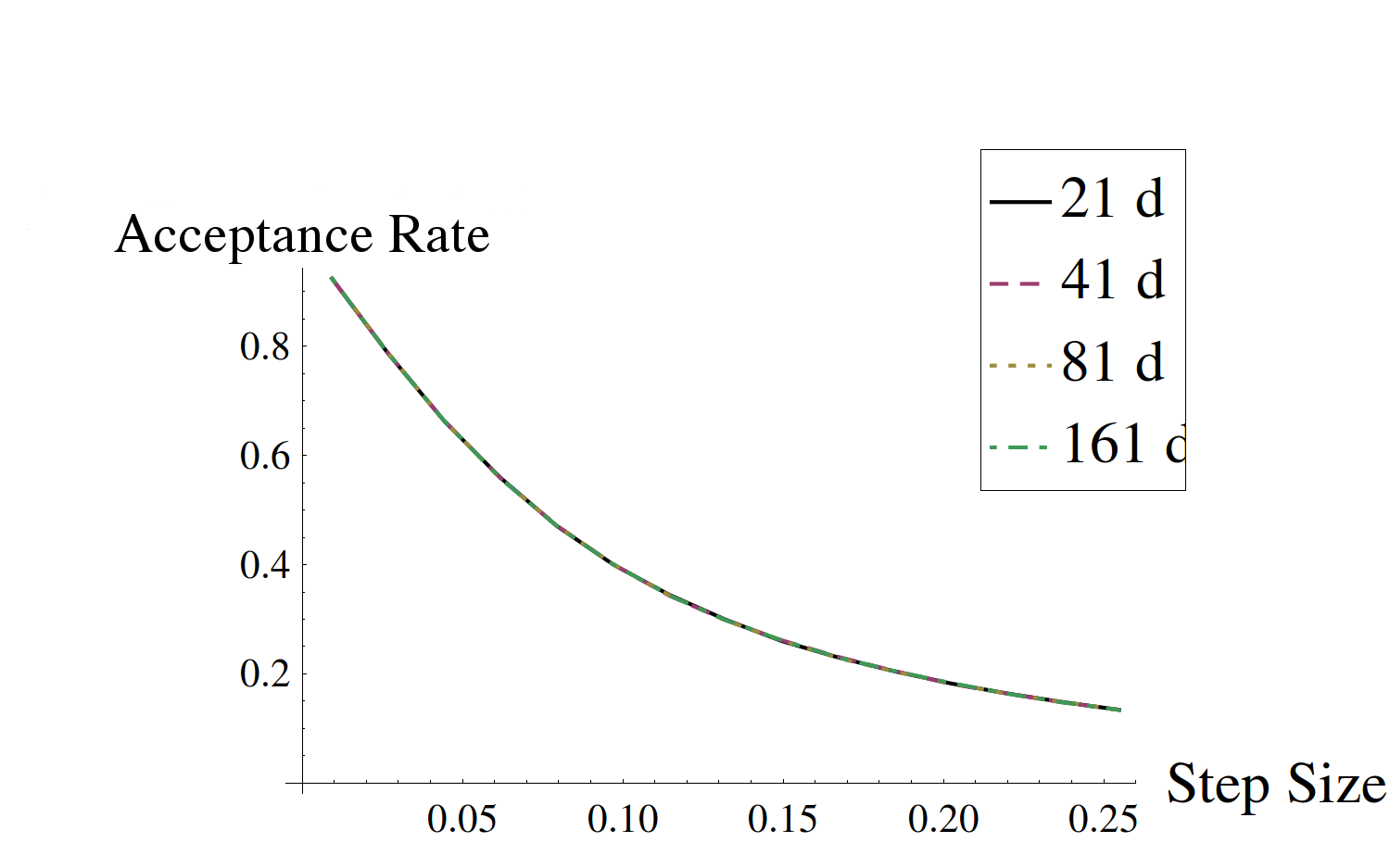}
}\\
\subfloat[Acceptance rate vs. step size for the RSRWM algorithm ]{\includegraphics[width=0.7\textwidth]{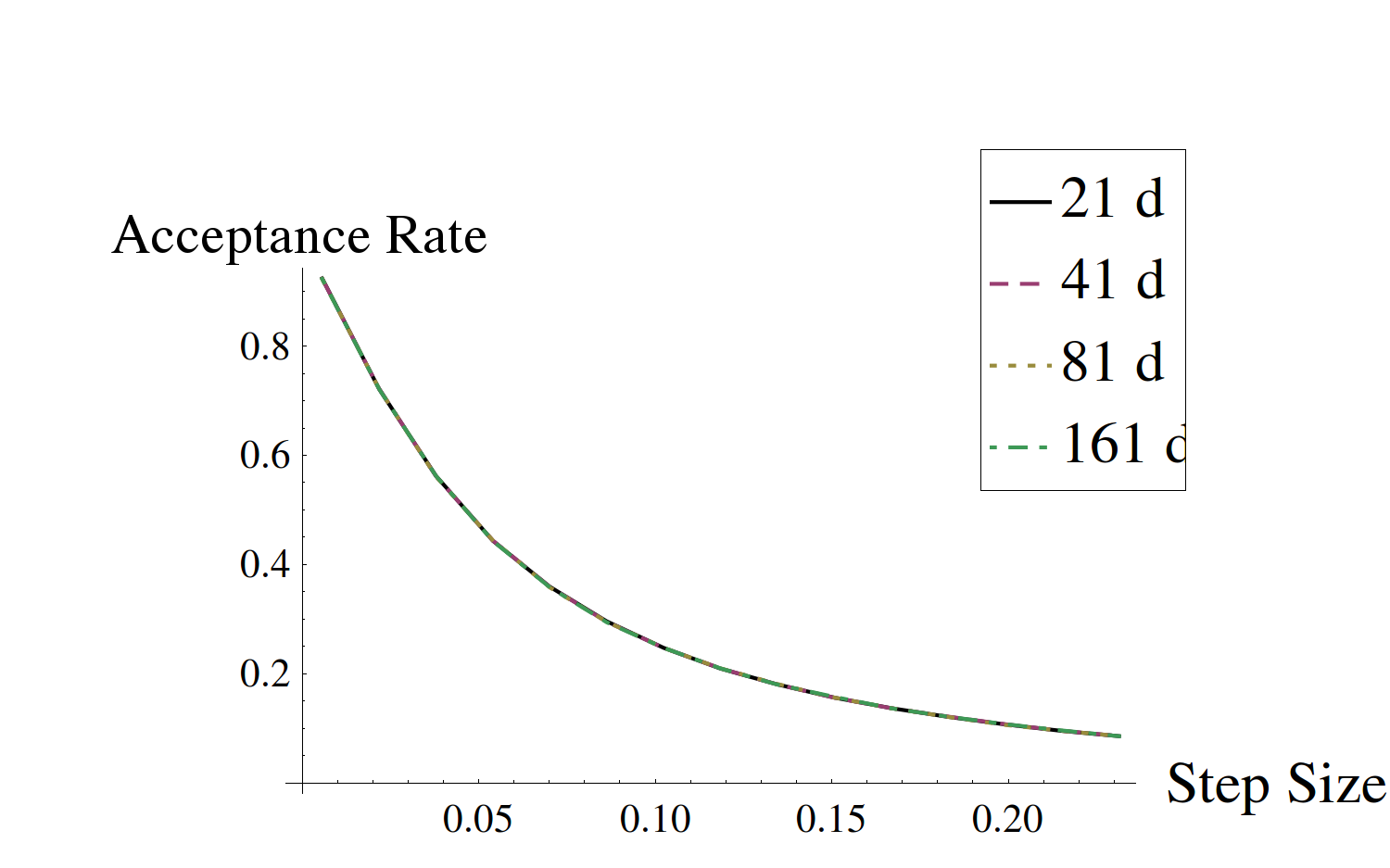}
}
\caption{\label{fig:acc}Dependence of the acceptance probability on the dimension}
\end{center}
\end{figure}

\subsection{Autocorrelation of the IS, the RWM, the RURWM and the RSRWM Algorithm}
Even though our  lower bound on the $L^2_\mu$-spectral gap is smaller for the RRWM algorithm than for the IS algorithm (cf. Remark \ref{rem:independenceSampler1}), the  numerical results in this section suggest that the RRWM algorithm
outperforms the IS algorithm especially if $\target$ is peaked. The peakedness of $\target$ is achieved by observing $\pressure$ on a fine mesh with small noise ($\disObs x=0.03$ and $\sigma=0.03$). 

The computational cost of both algorithms is nearly the same because
the cost of computing the likelihood is more expensive than generating
the proposal, which is slightly more expensive for the RRWM algorithm.
Subsequently, we compare the RWM, the IS, the RURWM and the RSRWM algorithm by plotting their autocorrelation. We consider $K=25$ ($K=250$) corresponding to an expansion with $25$ ($250$) sine and $25$ ($250$) cosine coefficients and a constant term thus giving rise to a 51 ($501$) dimensional problem.

In order to compare the RWM and the RRWM algorithms in a fair way, we choose the step size $\epsilon$ in a way to get an acceptance rate close to $0.135$. In the case of the RWM algorithm this is motivated by the optimal scaling results  in \cite{MR3025684}. The optimality of this acceptance rate is indicated by proving that the properly rescaled samples converge to a Langevin diffusion whose time scale depends on the acceptance rate of the RWM algorithm. An acceptance rate of $0.135$ corresponds to the largest time scale and thus to a quicker convergence to equilibrium of the Langevin diffusion. For the RRWM algorithm the acceptance rate is not affected by the choice of $J$. However, it is reasonable to choose a step size with acceptance probability bounded away from one and zero. 

For the lazy version of the RRWM algorithm we know that the $L^2_\mu$-spectral gap is bounded below and thus the asymptotic variance of the CLT (c.f. Proposition \ref{kipnisvara-clt}) for $f \in L^2_\target$ is bounded above.  The asymptotic variance can be related to the autocorrelation which is given by
\[
	c_i = \text{Cov}(f(X_0),f(X_i))
\] where $X_i$  is the evolution of the corresponding Markov chain. It is well-known that the asymptotic variance is equal to the integrated autocorrelation \cite{Rosenthal2007variance,Meyn:2009uqa} which is given by 
\[
	\sigma^2 = c_0+2 \sum_{i=0}^\infty  c_i.
\]
We consider the Markov chain resulting from the IS, the RWM, the RURWM and the RSRWM algorithm on the state space $[-1,1]^{J+1}$. We denote by  $u_i$ the $i=0,\dots,J$ projections onto the $i+1$-th coordinate. In the following we consider the autocorrelation for $u_0$ (c.f. Equation \ref{eq:simulations}) for the algorithms mentioned above.

Simulations for $d=0.1$ and $\sigma =0.1$ are presented in Figure \ref{fig:autocorr1} which shows that the autocorrelation of the RURWM, the RSRWM and the IS algorithm is only affected up to a point by the dimension of the state space. In contrast, the autocorrelation of the RWM decays much slower for the $501$-dimensional state space as for the $51$-dimensional state space.
In Figure \ref{fig:autocorr2}, we consider the decay of the autocorrelation of the IS, the RWM, the RURWM and the RSRWM algorithm for  more observations and lower observational noise ($d=0.04$ and $\sigma =0.03$). However, this implies that the measure concentrates in smaller regions of the state space making it harder to sample from.
Figure \ref{fig:autocorr2} illustrates that the RURWM and the RSRWM algorithm can be tuned to work well for concentrated target measures whereas the IS algorithm behaves poorly even though it is dimension independent.

For a fixed step size the RWM algorithm deteriorates as the dimension increases because the probability that one component steps outside $[-1,1]$ converges to one. If the step size is scaled to zero appropriately, the performance of the RWM algorithm deteriorates slower but for a large enough state space the IS algorithm outperforms the RWM algorithm for fixed observation operator and observational noise. The reason for this is that Corollary \ref{thm:application} yields a dimension independent lower bound on the performance of the IS, the RURWM and the RSRWM algorithm.

\begin{figure}
\begin{center}
\subfloat[51-dimensional state space, acceptance rate: RWM 14.1\%, IS 3.9\%,
RURWM 24.8\% RSRWM 27.7\% ]{\includegraphics[width=0.46\textwidth]{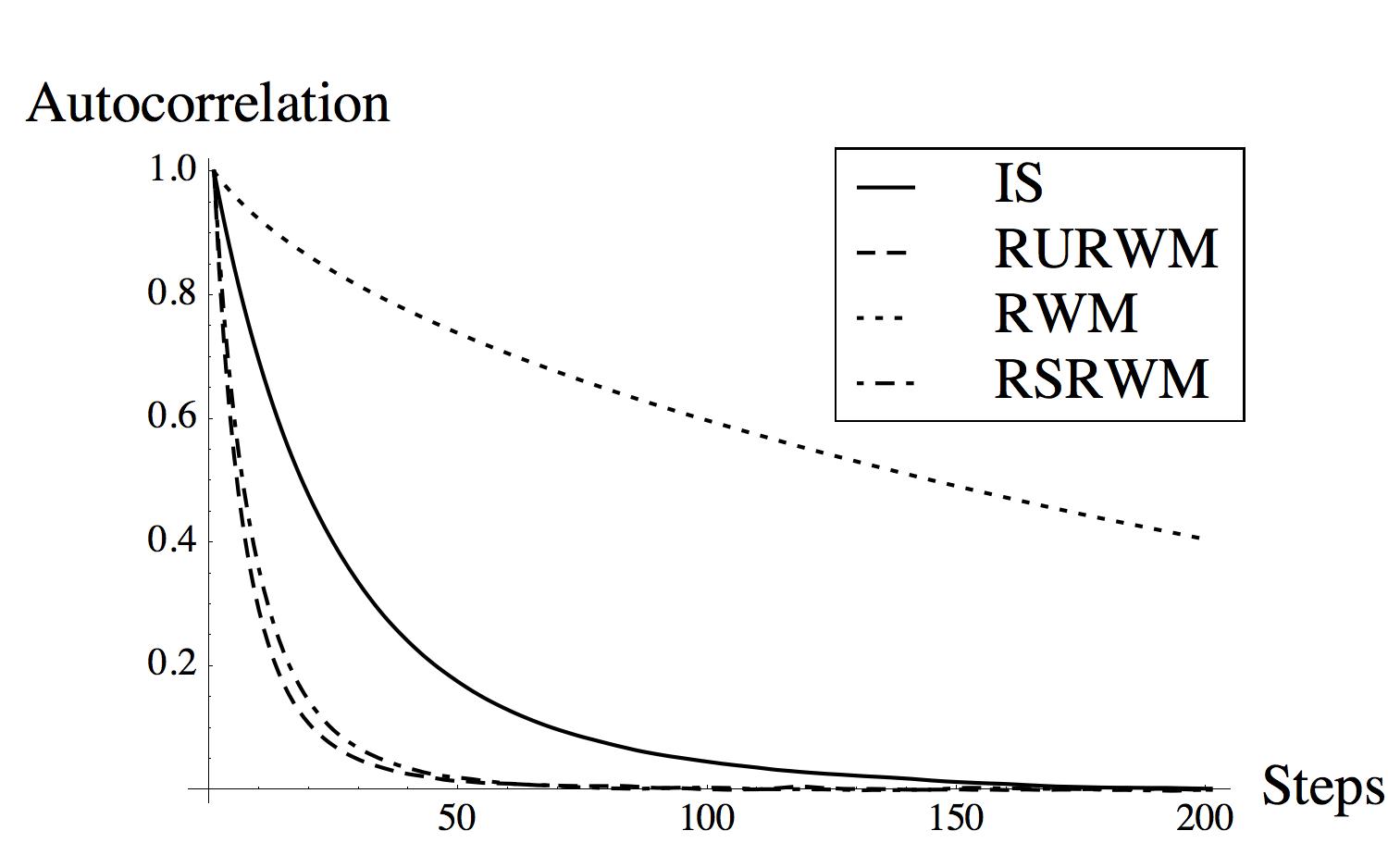}
}\hspace{0.05\textwidth}
\subfloat[501-dimensional state space, acceptance rate: IS 4.6\%, RWM 14.4\%,
RURWM 26.6\%, RSRWM 28.7\% ]{\includegraphics[width=0.46\textwidth]{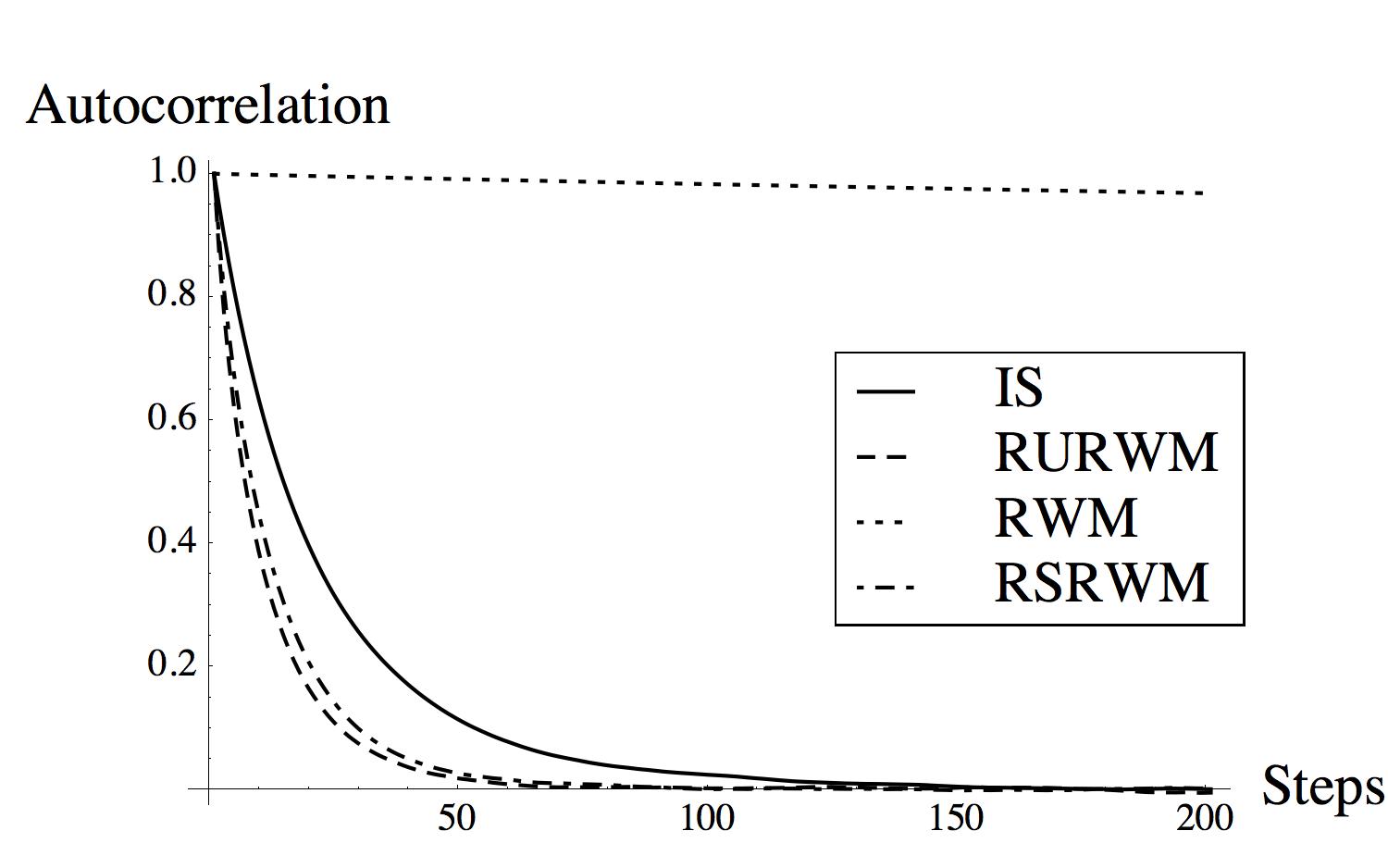}
}
\end{center}
\caption{\label{fig:autocorr1}Autocorrelation arising from a posterior for $\sigma$=0.1 and $\disObs=0.1$ }
\end{figure}

\begin{figure}
\begin{center}

\subfloat[51-dimensional state space, acceptance rate: IS 0.0001\%, RWM 25.4\%, RURWM 21\%,RSRWM 24.6\%]{\includegraphics[width=0.46\textwidth]{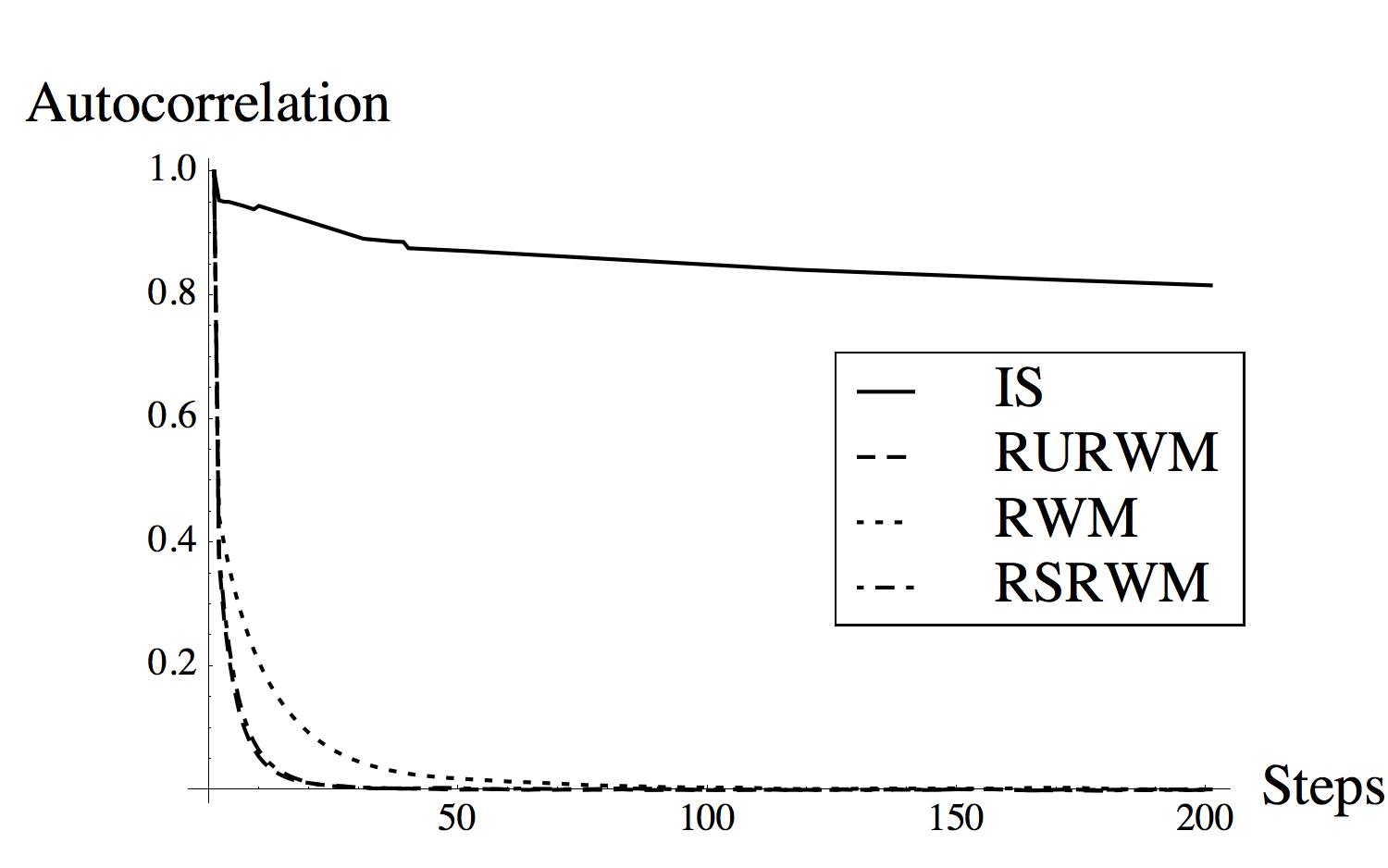}
}\hspace{0.05\textwidth}
\subfloat[501-dimensional state space, acceptance rate: IS 0.04\%, RWM 14.2\%, 
RURWM 30\%, RSRWM 25.5\%]{\includegraphics[width=0.46\textwidth]{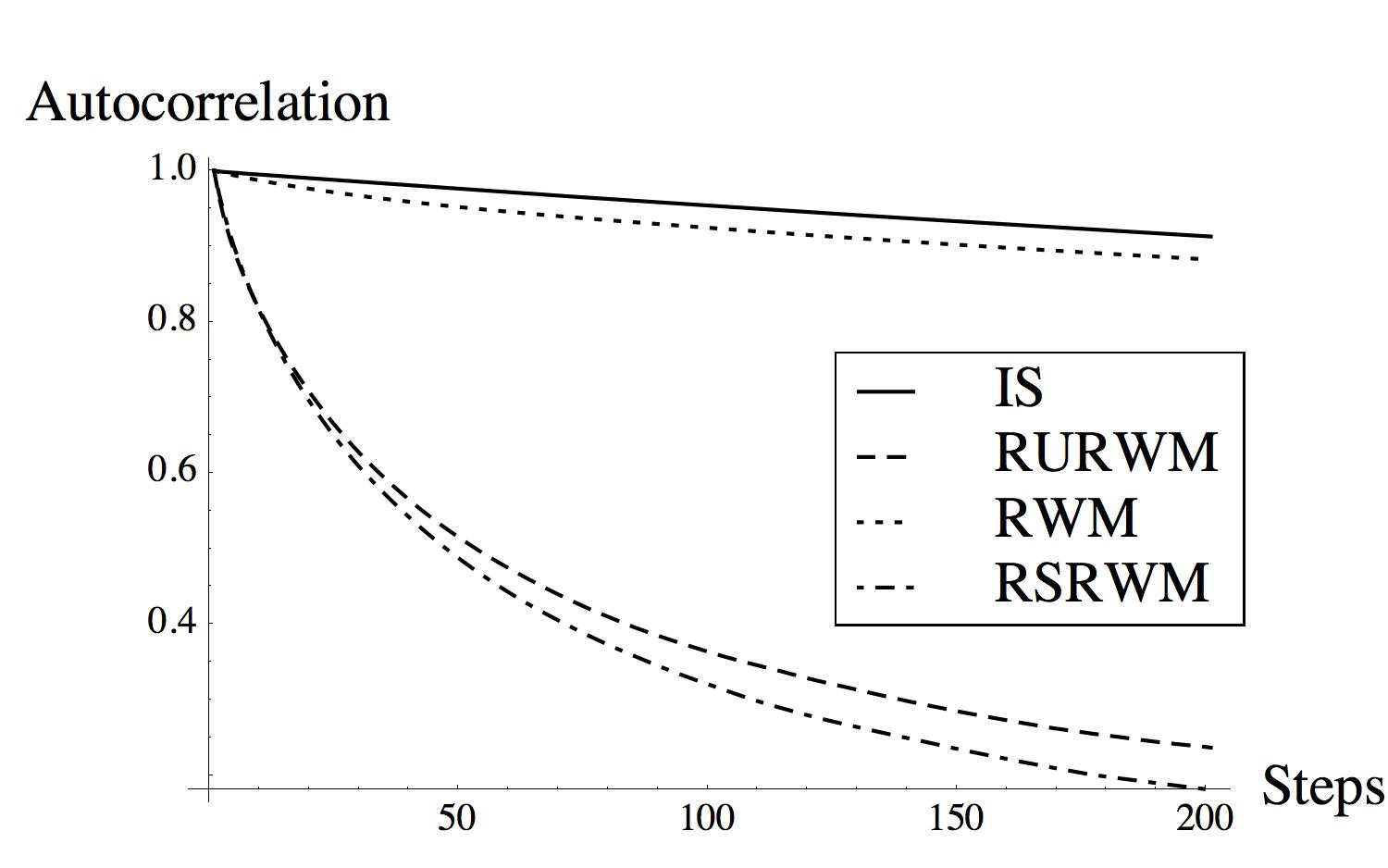}
}\\
\end{center}
\caption{\label{fig:autocorr2}Autocorrelation arising from posterior for $\sigma$=0.05 and $\disObs=0.05$}
\end{figure}

\section{Conclusion and Avenues of Further Research}

In this article, we have shown that it is possible to transfer $L^{2}$-spectral gaps from the proposal
Markov kernel to the lazy version of the Metropolis-Hastings Markov kernel. This yields
theoretical bounds for a large class of proposals for non-Gaussian
measures on function spaces. Our main assumption is that the density with respect to the reference measure is bounded above and below. This is a restrictive condition but it is difficult to prove any results in great generality
under weaker assumptions. The assumption that the density is bounded above and below on bounded sets seems weak enough. Both assumptions
only differ in the tails and restricting the problem to a large enough set decreases the probability of a sampling algorithm leaving it in the duration of a simulation to almost zero. But it is often the tail behaviour
which prevents algorithms from satisfying the desired convergence properties, see for example \cite{roberts1996exponential} which describes the phenomenon for the Langevin diffusion. This effect is also described in \cite{newMALA}, but it is not clear what impact this behaviour has on the sample average. 

Our main result justifies the use of sampling methods other than the IS algorithm for the
Bayesian elliptic inverse problem considered above. However, our bounds do not show that locally moving algorithms, as the RURWM and the RSRWM algorithm designed in Section 4, are asymptotically better
than the IS algorithm. Comparing two sampling algorithms is difficult since their performance depend on the specific target. Moreover, the performance also depends on the choice of the parameters, for example the step size of the algorithms. Nonetheless, rigorously  showing that the RURWM and the RSRWM algorithm outperform the IS algorithm, even in a special case, would be an interesting  result.

Moreover, the range of the posterior density goes to infinity as the variance of the noise goes to zero. This suggests that sampling methods perform worse and worse as the observational noise goes to zero. Getting precise asymptotics of this behaviour would lead to a better understanding of the performance of sampling algorithms for Bayesian inverse problems.

As mentioned in Section \ref{sec:Application}, the proposal kernels of the RURWM and the RSRWM algorithm are based on a tensorisation of Markov kernels for the uniform distribution on $[-1,1]$. It is also interesting to consider tensorisation of Metropolis-Hastings kernels   for the uniform distribution on $[-1,1]$.
Whereas we used the explicit structure of the prior, an interesting direction for more complicated priors is to use Metropolis-Hastings chains or combinations, such as tensorisation. This can lead to good proposals for another Metropolis-Hastings chain. Note that even if some of the  Metropolis-Hastings algorithms in the tensorisation reject their proposal, the overall proposal can still be accepted. A deeper investigation of this approach can lead to a better  understanding and guidelines  for the design of efficient proposals.  An interesting special case are MCMC algorithms for Bayesian inverse problems formulated on the coefficients of a Fourier series expansion. Usually the coefficients corresponding to high frequencies have only little impact on the forward problem and hence the inverse problem.  Developing proposals that exploit this phenomenon should also be pursued. 

One extension of this application which is of particular interest would be to consider a multi-scale diffusion coefficient because there is interest in the fine and coarse scale properties of the permeability for example in subsurface geophysics. Homogenisation results imply that different combinations of fine and coarse scales lead to effectively the same homogenised problem thus leading to a lack of identifiability. This also seems to be a very interesting idea.

\section*{Acknowledgements}
The author would like to thank Professor Andreas Eberle, Professor Martin Hairer and Professor Andrew Stuart for helpful discussions.  SJV is grateful for the support of an ERC scholarship.

 \appendix
 \section{Proof of Theorem \ref{thm:URMHproposal}}\label{sec:appendixProof}
We first prove that for $\epsilon$ small enough there is an
$\epsilon$-independent lower bound on the $L_{\prior}^{2}$-spectral
gap of $\left(Q_{\epsilon}^{\text{RURWM}}\right)^{n}$ with $n(\epsilon)=\left\lceil \frac{1}{\epsilon^{2}}\right\rceil $.
This is achieved by showing that $[-1,1]$ is a small set for $\left(Q_{\epsilon}^{\text{RURWM}}\right)^{n}$.
This implies uniform ergodicity and in turn a lower bound on the $L_{\mathcal{U}(-1,1)}^{2}$-spectral
gap of $\left(Q_{\epsilon}^{\text{RURWM}}\right)^{n}$. A lower bound
on the $L_{\mathcal{U}(-1,1)}^{2}$-spectral gap of $Q_{\epsilon}^{\text{RURWM}}$
can then be obtained using the spectral theorem. By $\tilde{q}_{\epsilon}(x,y)=\mathbbm{1}_{(x-\epsilon,x+\epsilon)}(y)$
we denote the density of the unreflected random walk. The density
$q_{\epsilon,n}^{\text{RURWM}}$ of $\left(Q_{\epsilon}^{\text{RURWM}}\right)^{n}$
is point-wise larger than $\tilde{q}_{\epsilon,n}$ because each $y$
might have several preimages under $R$ (c.f. Equation (\ref{eq:reflection})).
In order to show that $[-1,1]$ is a small set, we need to obtain a
uniformly lower bound on the transition density $\tilde{q}_{\epsilon,n}$.
This is achieved using a local limit theorem from \cite{petrov1975sumiid}.\\
\begin{theorem}
(Theorem 13 in Section 7 of \cite{petrov1975sumiid}) Let $\left\{ X_{n}\right\} $
be a sequence of independent random variables having a common distribution
with zero mean, non-zero variance, and finite absolute moment $\mathbb{E}\left|X_{1}\right|^{k}$
of some integer order $k\ge3$. Let  $p_{N}(x)$ be the density of the random variable $\frac{1}{\sigma\sqrt{n}}\sum_{j=1}^{n}X_{j}$. Then 
\[
p_{n}(x)=\phi(x)+\sum_{v=1}^{k-2}\frac{q_{\nu}(x)}{n^{\nu/2}}+o\left(\frac{1}{n^{(k-2)/2}}\right)
\]
 where $\phi(x)=\frac{1}{\sqrt{2\pi}}\exp\left(-\frac{x^{2}}{2}\right)$
is the density of $\mathcal{N}(0,1).$\\
\end{theorem}

For our case we are able to obtain the following corollary.\\
\begin{corollary}
\label{cor:unilocal}Let $U_{n}\overset{\text{i.i.d.}}{\sim}\mathcal{U}(-1,1)$
then the density of $p_{n}$ of $\frac{1}{\sqrt{n/3}}\sum_{j=1}^{n}U_{j}$
satisfies
\[
p_{n}(x)=\phi(x)+O\left(\frac{1}{n}\right).
\]
 \end{corollary}We denote the probability density of $\epsilon\sum_{i=1}^{n}U_{i}$
by $\tilde{p}_{n}^{\epsilon}(x)$ which is related to $p_{n}$ through
\[
\tilde{p}_{n}^{\epsilon}(x)=p_{n}^{\epsilon}\left(\frac{x}{\epsilon\sqrt{\frac{1}{3}n}}\right)\frac{1}{\sqrt{\frac{1}{3}n}\epsilon}.
\]
Using $n(\epsilon)=\left\lceil \frac{1}{\epsilon^{2}}\right\rceil $
and Corollary \foreignlanguage{british}{\ref{cor:unilocal}, we know
that} 
\begin{eqnarray*}
\left|\tilde{p}_{n(\epsilon)}^{\epsilon}(x)-\phi\left(\frac{x}{\epsilon\sqrt{\frac{1}{3}n(\epsilon)}}\right)\frac{1}{\epsilon\sqrt{\frac{1}{3}n(\epsilon)}}\right| & \leq & \frac{1}{\epsilon\sqrt{\frac{1}{3}n(\epsilon)}}\left|p_{n}\left(\frac{x}{\epsilon\sqrt{\frac{1}{3}n(\epsilon)}}\right)-\phi\left(\frac{x}{\epsilon\sqrt{\frac{1}{3}n(\epsilon)}}\right)\right|\\
 & \leq & \sqrt{3}C\frac{1}{n(\epsilon)}.
\end{eqnarray*}
Since $p_{n}$ is symmetric and log-concave \[\inf_{x\in[-2,2]}\tilde{p}_{n(\epsilon)}(x)=\tilde{p}_{n(\epsilon)}^{\epsilon}(2),\]
the aim is to obtain a lower bound on $\tilde{p}_{n}^{\epsilon}(2).$
This is achieved by noting that
\[
\phi\left(\frac{2}{\epsilon\sqrt{\frac{1}{3}n(\epsilon)}}\right)\frac{1}{\epsilon\sqrt{\frac{1}{3}n(\epsilon)}}\ge\phi\left(2\sqrt{3}\right)\frac{\sqrt{3}}{2}=:2l.
\]
For all $\epsilon\leq\epsilon_{0}$ small enough and hence $n(\epsilon)$
large enough
\[
\left|\tilde{p}_{n(\epsilon)}^{\epsilon}(x)-\phi\left(\frac{x}{\epsilon\sqrt{\frac{1}{3}n(\epsilon)}}\right)\frac{1}{\epsilon\sqrt{\frac{1}{3}n(\epsilon)}}\right|\leq l\quad\forall x.
\]
Using the triangle inequality, this yields a uniform lower bound on
the transition kernel 
\[
\tilde{q}_{\epsilon,n}(x,y)\ge\tilde{p}_{n(\epsilon)}^{\epsilon}(2)\ge\phi\left(\frac{2}{\epsilon\sqrt{\frac{1}{3}n(\epsilon)}}\right)\frac{1}{\epsilon\sqrt{\frac{1}{3}n(\epsilon)}}-l\ge l\quad\forall x,y\in[-1,1].
\]
Therefore $q_{\epsilon,n}^{\text{RURWM}}$ also satisfies 
\[
q_{\epsilon,n}^{\text{RURWM}}(x,y)\ge\tilde{q}_{\epsilon,n}(x,y)\ge l\quad\forall x,y\in[-1,1].
\]
Thus, the state space $[-1,1]$ is a small set and hence we can apply Theorem 8 in \cite{roberts2004general}
which implies that 
\[
\left\Vert \left(Q_{\epsilon}^{\text{RURWM}}\right)^{n\cdot k}(d,dy)-\mathcal{U}(-1,1)\right\Vert _{\text{TV}}\leq(1-l)^{k}.
\]
For reversible Markov processes uniform ergodicity implies an $L_{\mathcal{U}(-1,1)}^{2}$-spectral
gap of the same size, see for example \cite{explicitbdd}. Hence $\left(Q_{\epsilon}^{\text{RURWM}}\right)^{n}$
has an $L_{\mathcal{U}(-1,1)}^{2}$-spectral gap of size 1-$\hat{\beta}=l$.
The $L_{\mathcal{U}(-1,1)}^{2}$-spectral theorem for self-adjoint
operators now implies that the $L_{\mathcal{U}(-1,1)}^{2}$-spectral
gap of $Q_{\epsilon}^{\text{RURWM}}$ is
\[
1-\beta_{\epsilon}=1-(1-l)^{\frac{1}{n}}\ge\frac{l}{n}\ge\frac{l}{2}\epsilon^{2}.
\]
It is left to treat $1\ge\epsilon>\epsilon_{0}$.  For this range of $\epsilon$
we choose $n=n(\epsilon_{0})=\left\lceil \frac{1}{\epsilon_{0}^{2}}\right\rceil $ so that
\begin{align*}
\left| \tilde{p}_{n(\epsilon_{0})}^{\epsilon} \left(x\right) -\phi\left(\frac{x}{\epsilon\sqrt{\frac{1}{3}n(\epsilon_{0})}}\right) \frac{1}{\epsilon\sqrt{\frac{1}{3}n(\epsilon_{0})}}
 \right| \\
 =\left|p_{n(\epsilon_{0})}\left(\frac{x}{\epsilon\sqrt{\frac{1}{3}n(\epsilon_{0})}}\right)\frac{1}{\epsilon\sqrt{\frac{1}{3}n(\epsilon_{0})}}-\phi\left(\frac{x}{\epsilon\sqrt{\frac{1}{3}n(\epsilon_{0})}}\right)\frac{1}{\epsilon\sqrt{\frac{1}{3}n(\epsilon_{0})}}\right| \\
 \leq & \frac{\sqrt{3}\epsilon_{0}}{\epsilon}C\frac{1}{n(\epsilon_{0})}\leq  \frac{\epsilon_{0}}{\epsilon}l
\end{align*}

On the other hand
\[
\phi\left(\frac{2}{\epsilon\sqrt{\frac{1}{3}n(\epsilon_{0})}}\right)\frac{1}{\epsilon\sqrt{\frac{1}{3}n(\epsilon_{0})}}\ge\phi\left(\sqrt{3}2\frac{\epsilon_{0}}{\epsilon}\right)\frac{\epsilon_{0}}{\epsilon}\frac{\sqrt{3}}{2}\ge\frac{\epsilon_{0}}{\epsilon}2l.
\]
Similarly to the above, we know that
\[
q_{\epsilon,n}^{\text{RURWM}}(x,y)\ge\tilde{q}_{\epsilon,n}(x,y)\ge\frac{\epsilon_{0}}{\epsilon}l\quad\forall x,y\in[-1,1].
\]
Hence it follows that the $L_{\mathcal{U}(-1,1)}^{2}$-spectral
gap of $(Q_{\epsilon}^{\text{RURWM}})^{n(\epsilon_{0})}$ is bounded
below by $\frac{\epsilon_{0}}{\epsilon}l$. Using the spectral theorem,
we conclude that the $L_{\mathcal{U}(-1,1)}^{2}$-spectral gap $1-\beta_{\epsilon}$
of $Q_{\epsilon}^{\text{RURWM}}$ satisfies
\[
1-\beta_{\epsilon}\ge1-\left(1-\frac{\epsilon_{0}}{\epsilon}l\right)^{\frac{1}{n(\epsilon)}}\ge\frac{\epsilon_{0}}{\epsilon}l\frac{1}{n(\epsilon)}\ge\frac{\epsilon_{0}^{3}}{2\epsilon}l=\frac{\epsilon_{0}^{3}}{\epsilon^{3}}\frac{l}{2}\epsilon^{2}\ge\epsilon_{0}^{3}\frac{l}{2}\epsilon^{2}.
\]

\bibliographystyle{plain}
\bibliography{./../../../../docearcorrect}

\begin{thebibliography}{10}

\bibitem{Adler:2007fk}
Robert~J. Adler and Jonathan~E. Taylor.
\newblock {\em {R}andom {F}ields and {G}eometry}.
\newblock Springer, 2007.

\bibitem{1996AmitGibbs}
Y.~Amit.
\newblock {C}onvergence {P}roperties of the {G}ibbs {S}ampler for
  {P}erturbations of {G}aussians.
\newblock {\em Ann. Statist.}, 24(1):122--140, 1996.

\bibitem{babuska2004galerkin}
I.~Babuska, R.~Tempone, and G.~E. Zouraris.
\newblock Galerkin {F}inite {E}lement {A}pproximations of {S}tochastic
  {E}lliptic {P}artial {D}ifferential {E}quations.
\newblock {\em SIAM J. Numer. Anal.}, 42(2):800--825, 2004.

\bibitem{bakry2006functional}
Dominique Bakry.
\newblock Functional {I}nequalities for {M}arkov {S}emigroups.
\newblock In {\em Probability measures on groups: recent directions and
  trends}, pages 91--147. Tata Inst. Fund. Res., Mumbai, 2006.

\bibitem{1994BernardoBayesianBook}
J.~M. Bernardo and A.~F.~M. Smith.
\newblock {\em Bayesian Theory}.
\newblock Wiley, 1994.

\bibitem{gaussianMeasureas}
Vladimir~I. Bogachev.
\newblock {\em {G}aussian {M}easures}, volume~62 of {\em Mathematical Surveys
  and Monographs}.
\newblock Amer. Math. Soc., Providence, RI, 1998.

\bibitem{newMALA}
N.~Bou-Rabee and M.~Hairer.
\newblock {N}on-asymptotic mixing of the {M}{A}{L}{A} algorithm.
\newblock {\em IMA J. Numer. Anal.}, (33):pp. 80--110, 2013.

\bibitem{brooks2011handbook}
Steve Brooks, Andrew Gelman, Galin Jones, and Xiao-Li Meng.
\newblock {\em {H}andbook of {M}arkov {C}hain {M}onte {C}arlo}.
\newblock Chapman and Hall/CRC, 2011.

\bibitem{2010SchwabEllipticUQ}
A.~Cohen, R.~A. Devore, and C.~Schwab.
\newblock Convergence rates of best {$N$}-term {G}alerkin approximations for a
  class of elliptic s{PDE}s.
\newblock {\em Found. Comput. Math.}, 10(6):615--646, 2010.

\bibitem{cotter-mcmc}
S.~L. Cotter, M.~Dashti, J.~C. Robinson, and A.~M. Stuart.
\newblock Bayesian {I}nverse {P}roblems for {F}unctions and {A}pplications to
  {F}luid {M}echanics.
\newblock {\em Inverse Probl.}, 25(11):115008, 43, 2009.

\bibitem{Con3}
M.~Dashti, S.~Harris, and A.~M. Stuart.
\newblock {B}esov {P}riors for {B}ayesian {I}nverse {P}roblems.
\newblock {\em Inverse Probl. Imaging}, 6:183--200, 2012.

\bibitem{2013dashtiMap}
M.~Dashti, K.~J.~H. Law, A.~M. Stuart, and J.~Voss.
\newblock {M}{A}{P} {E}stimators and {P}osterior {C}onsistency in {B}ayesian
  {N}onparametric {I}nverse {P}roblems.
\newblock {\em ArXiv preprint 1303.4795}, 2013.

\bibitem{UncertaintyElliptic}
M.~Dashti and A.~M. Stuart.
\newblock Uncertainty {Q}uantification and {W}eak {A}pproximation of an
  {E}lliptic {I}nverse {P}roblem.
\newblock {\em SIAM J. Numer. Anal.}, 49:2524--2542, 2011.

\bibitem{diaconis1993comparision}
P.~Diaconis and L.~Saloff-Coste.
\newblock Comparison theorems for reversible {M}arkov chains.
\newblock {\em Ann. Appl. Probab.}, 3(3):696--730, 1993.

\bibitem{Fukushima2011DirichletForms}
Masatoshi Fukushima, Yoichi Oshima, and Masayoshi Takeda.
\newblock {\em Dirichlet forms and symmetric {M}arkov processes}, volume~19 of
  {\em de Gruyter Studies in Mathematics}.
\newblock Walter de Gruyter \& Co., Berlin, extended edition, 2011.

\bibitem{ghanem2003stochastic}
Roger~G. Ghanem and Pol~D. Spanos.
\newblock {\em Stochastic finite elements: a spectral approach}.
\newblock Courier Dover Publications, 2003.

\bibitem{guionnet1801lectures}
Alice Guionnet and Boguslaw Zegarlinski.
\newblock {\em Lectures on {L}ogarithmic {S}obolev {I}nequalities}, volume 1801
  of {\em S{\'e}minaire de Probabilit{\'e}s, XXXVI}.
\newblock Springer, 2002.

\bibitem{Rosenthal2007variance}
O.~H{\"a}ggstr{\"o}m and J.~S. Rosenthal.
\newblock On {V}ariance {C}onditions for {M}arkov {C}hain {C}{L}{T}s.
\newblock {\em Electron. Comm. Probab.}, 12:454--464 (electronic), 2007.

\bibitem{hairer2011spectral}
M.~Hairer, A.~M. Stuart, and S.~J. Vollmer.
\newblock {Spectral Gaps for a Metropolis-Hastings Algorithm in Infinite
  Dimensions}.
\newblock {\em ArXiv preprint 1112.1392}, 2011.

\bibitem{2013HoangComplexityMCMC}
V.~H. {Hoang}, C.~{Schwab}, and A.~M. {Stuart}.
\newblock {C}omplexity {A}nalysis of {A}ccelerated {M}{C}{M}{C} {M}ethods for
  {B}ayesian {I}nversion.
\newblock {\em ArXiv e-prints}, July 2012.

\bibitem{Jarner2000341}
S.~F. Jarner and E.~Hansen.
\newblock Geometric {E}rgodicity of {M}etropolis {A}lgorithms.
\newblock {\em Stoch. Proc. Appl.}, 85(2):341 -- 361, 2000.

\bibitem{2006JonesFixedWidth}
G.~L. Jones, M.~Haran, B.~S. Caffo, and R.~Neath.
\newblock Fixed-width output analysis for {M}arkov chain {M}onte {C}arlo.
\newblock {\em J. Amer. Statist. Assoc.}, 101(476):1537--1547, 2006.
\newblock confidenceMCMC.

\bibitem{MR2102218}
Jari Kaipio and Erkki Somersalo.
\newblock {\em Statistical and {C}omputational {I}nverse {P}roblems}, volume
  160 of {\em Applied Mathematical Sciences}.
\newblock Springer-Verlag, New York, 2005.

\bibitem{kipnis1986central}
C.~Kipnis and S.~R.~S. Varadhan.
\newblock {C}entral {L}imit {T}heorem for {A}dditive {F}unctionals of
  {R}eversible {M}arkov {P}rocesses and {A}pplications to {S}imple
  {E}xclusions.
\newblock {\em Comm. Math. Phys.}, 104(1):1--19, 1986.

\bibitem{1995KunischNumericalEIP}
K.~Kunisch.
\newblock Numerical {M}ethods for {P}arameter {E}stimation {P}roblems {I}nverse
  {P}roblems in {D}iffusion {P}rocesses.
\newblock In {\em Proc. GAMM-SIAM Symp. (Philadelphia, PA: SIAM)}, pages
  199--216, 1995.

\bibitem{Lasanen2007}
S~Lasanen.
\newblock {Measurements and infinite-dimensional statistical inverse theory}.
\newblock {\em PAMM}, 1080102:1080101--1080102, 2007.

\bibitem{lasanen2002Disc}
Sari Lasanen.
\newblock Discretizations of generalized random variables with applications to
  inverse problems.
\newblock {\em Ann. Acad. Sci. Fenn. Math. Diss.}, (130):64, 2002.
\newblock Dissertation, University of Oulu, Oulu, 2002.

\bibitem{2011KrysConfidence}
K.~{\L}atuszy{\'n}ski and W.~Niemiro.
\newblock Rigorous {C}onfidence {B}ounds for {M}{C}{M}{C} under a {G}eometric
  {D}rift {C}ondition.
\newblock {\em J. Complexity}, 27(1):23--38, 2011.

\bibitem{Latuszynskiclt}
K.~{\L}atuszy{\'n}ski and G.~O. Roberts.
\newblock {{C}{L}{T}s} and {A}symptotic {V}ariance of {T}ime-{S}ampled {M}arkov
  {C}hains.
\newblock {\em Methodol. Comput. Appl. Probab.}, pages 1--11, 2011.

\bibitem{levin2009markov}
David~Asher Levin, Yuval Peres, and Elizabeth~Lee Wilmer.
\newblock {\em Markov {C}hains and {M}ixing {T}imes}.
\newblock AMS Bookstore, 2009.

\bibitem{lovasz1998mixing}
L{\'a}szl{\'o} Lov{\'a}sz and Peter Winkler.
\newblock Mixing times.
\newblock {\em Microsurveys in discrete probability}, 41:85--134, 1998.

\bibitem{2012MattinglyDLimit}
J.~C. Mattingly, N.~S. Pillai, and A.~M. Stuart.
\newblock Diffusion {L}imits of the {R}andom {W}alk {M}etropolis {A}lgorithm in
  {H}igh {D}imensions.
\newblock {\em Ann. Appl. Probab.}, 22(3):881--930, 2012.

\bibitem{mclaughlin1996reassessment}
D.~McLaughlin and L.~R. Townley.
\newblock A {R}eassessment of the {G}roundwater {I}nverse {P}roblem.
\newblock {\em Water Resour. Res.}, 32(5):1131--1161, 1996.

\bibitem{Mengersen1996ConvergenceMCMC}
K.~L. Mengersen and R.~L. Tweedie.
\newblock Rates of {C}onvergence of the {H}astings and {M}etropolis
  {A}lgorithms.
\newblock {\em Ann. Statist.}, 24(1):101--121, 1996.

\bibitem{Meyn:2009uqa}
Sean Meyn and Richard~L. Tweedie.
\newblock {\em {M}arkov {C}hains and {S}tochastic {S}tability}.
\newblock Cambridge University Press, Cambridge, second edition, 2009.
\newblock With a prologue by Peter W. Glynn.

\bibitem{MR3025684}
P.~Neal, G.~O. Roberts, and W.~K. Yuen.
\newblock Optimal scaling of random walk {M}etropolis algorithms with
  discontinuous target densities.
\newblock {\em Ann. Appl. Probab.}, 22(5):1880--1927, 2012.

\bibitem{niemiro2009MEDIAN}
W.~Niemiro and P.~Pokarowski.
\newblock Fixed {P}recision {M}{C}{M}{C} {E}stimation by {M}edian of {P}roducts
  of {A}verages.
\newblock {\em J. Appl. Probab.}, 46(2):309--329, 2009.

\bibitem{petrov1975sumiid}
V.~V. Petrov.
\newblock {\em Sums of independent random variables}.
\newblock Springer-Verlag, New York, 1975.
\newblock Translated from the Russian by A. A. Brown, Ergebnisse der Mathematik
  und ihrer Grenzgebiete, Band 82.

\bibitem{Roberts1997Hybrid}
G.~O. Roberts and J.~S. Rosenthal.
\newblock {G}eometric {E}rgodicity and {H}ybrid {M}arkov {C}hains.
\newblock {\em Electron. Comm. Probab}, 2:13--25, 1997.

\bibitem{roberts2004general}
G.~O. Roberts and J.~S. Rosenthal.
\newblock {G}eneral {S}tate {S}pace {M}arkov {C}hains and {M}{C}{M}{C}
  {A}lgorithms.
\newblock {\em Probab. Surv.}, 1:20--71, 2004.

\bibitem{roberts1996exponential}
G.~O. Roberts and R.~L. Tweedie.
\newblock {E}xponential {C}onvergence of {L}angevin {D}istributions and their
  {D}iscrete {A}pproximations.
\newblock {\em Bernoulli}, pages 341--363, 1996.

\bibitem{Roberts2001GeomL2}
G.~O. Roberts and R.~L. Tweedie.
\newblock {G}eometric {L}2 and {L}1 {C}onvergence are {E}quivalent for
  {R}eversible {M}arkov {C}hains.
\newblock {\em J. Appl. Probab.}, 38(2001):37--41, 2001.

\bibitem{explicitbdd}
D.~Rudolf.
\newblock Explicit {E}rror {B}ounds for {M}arkov {C}hain {M}onte {C}arlo.
\newblock {\em Dissertationes Math. (Rozprawy Mat.)}, 485:1--93, 2012.

\bibitem{Schmuland1999Byron}
Byron Schmuland.
\newblock Dirichlet forms: some infinite-dimensional examples.
\newblock {\em Canad. J. Statist.}, 27(4):683--700, 1999.

\bibitem{schwab2011sparseUQelliptic}
C.~Schwab and C.~J. Gittelson.
\newblock Sparse {T}ensor {D}iscretizations of {H}igh-{D}imensional
  {P}arametric and {S}tochastic {P}{D}{E}s.
\newblock {\em Acta Numer.}, 20:291--467, 2011.

\bibitem{Con4}
C.~Schwab and A.~M. Stuart.
\newblock Sparse {D}eterministic {A}pproximation of {B}ayesian {I}nverse
  {P}roblems.
\newblock {\em Inverse Probl.}, 28(4):045003, 32, 2012.

\bibitem{MR2652785}
A.~M. Stuart.
\newblock {I}nverse {P}roblems: {A} {B}ayesian {P}erspective.
\newblock {\em Acta Numer.}, 19:451--559, 2010.

\bibitem{stuartchinanotes}
A.~M. {Stuart}.
\newblock {T}he {B}ayesian {A}pproach to {I}nverse {P}roblems.
\newblock {\em ArXiv preprint 1302.6989}, 2013.

\bibitem{samplingFirstInfiniteDimensional}
L.~Tierney.
\newblock {A} {N}ote on {M}etropolis-{H}astings {K}ernels for {G}eneral {S}tate
  {S}paces.
\newblock {\em Ann. Appl. Probab.}, 8(1):1--9, 1998.

\bibitem{2010WangErrorEIPReg}
L.~Wang and J.~Zou.
\newblock Error {E}stimates of {F}inite {E}lement {M}ethods for {P}arameter
  {I}dentification {P}roblems in {E}lliptic and {P}arabolic {S}ystems.
\newblock {\em Discrete Contin. Dyn. Syst. Ser. B}, 14:1641--1670, 2010.

\end{thebibliography}

\end{document}